\documentclass[11pt]{amsart}

\newtheorem{theorem}{Theorem}[section]

\newtheorem{lemma}[theorem]{Lemma}
\newtheorem{corollary}[theorem]{Corollary}

\title[Relations in universal Lie nilpotent associative algebras]{Relations in universal Lie nilpotent associative algebras of class 4}

\author{Eudes Antonio da Costa}

\address{Colegiado de Matem\' atica, Universidade Federal do Tocantins, 77330-0000, Arraias, TO, Brasil}

\email{eudes@uft.edu.br}

\author{Alexei Krasilnikov}

\address{Departamento de Matem\'atica, Universidade de Bras\'\i lia, 70910-900 Bras\'\i lia, DF, Brasil}

\email{alexei@unb.br}

\date{}

\begin{document}

\begin{abstract}
Let $K$ be a unital associative and commutative ring and let $K \langle X \rangle$ be the free unital associative  $K$-algebra on a non-empty set $X$ of free generators. Define a left-normed commutator $[a_1, a_2, \dots , a_n]$ inductively by $[a_1, a_2] = a_1 a_2 - a_2 a_1$, $[a_1, \dots , a_{n-1}, a_n] = [[a_1, \dots , a_{n-1}], a_n]$ $(n \ge 3)$. For $n \ge 2$, let $T^{(n)}$ be the two-sided ideal in $K \langle X \rangle$ generated by all commutators $[a_1,a_2, \dots , a_n]$ $( a_i \in K \langle X \rangle )$.

It can be easily seen that the ideal $T^{(2)}$ is generated (as a two-sided ideal in $K \langle X \rangle$) by the commutators $[x_1, x_2]$ $(x_i \in X)$. It is well-known that $T^{(3)}$ is generated by the polynomials $[x_1,x_2,x_3]$ and $[x_1,x_2][x_3,x_4] + [x_1,x_3][x_2,x_4]$ $(x_i \in X)$. A similar generating set for $T^{(4)}$ contains 3 types of polynomials in $x_i \in X$ if $\frac{1}{3} \in K$ and 5 types if $\frac{1}{3} \notin K$. In the present article we exhibit a generating set for $T^{(5)}$ that contains 8 types of polynomials in $x_i \in X$.
\end{abstract}

\maketitle

\section{Introduction}

Let $K$ be a unital associative and commutative ring and let $A$ be a unital associative $K$-algebra. Define a left-normed commutator $[a_1, a_2, \dots , a_n]$ inductively by $[a_1, a_2] = a_1 a_2 - a_2 a_1$, $[a_1, \dots , a_{n-1}, a_n] = [[a_1, \dots , a_{n-1}], a_n]$ $(n \ge 3)$. For $n \ge 2$, let $T^{(n)} (A)$ be the two-sided ideal in $A$ generated by all commutators $[a_1,a_2, \dots , a_n]$ $( a_i \in A )$. An algebra $A$ is called \textit{Lie nilpotent} (of class at most $n-1$) if $T^{(n)} (A) = 0$ for some $n \ge 2$.

Let $K \langle X \rangle$ be the free unital associative $K$-algebra on a non-empty set $X$ of free generators. Define $T^{(n)} = T^{(n)} ( K \langle X \rangle )$.  The quotient algebra $K \langle X \rangle/T^{(n)}$ can be viewed as the universal Lie nilpotent associative $K$-algebra of class $n-1$ generated by $X$.

The study of Lie nilpotent associative rings and algebras was started by Jennings \cite{Jennings47} in 1947. Since then Lie nilpotent associative rings and algebras have been investigated in various papers from various points of view; see, for instance, \cite{AS01, Gordienko07, GP15,GL83, Kras97, Latyshev65, RT99, Volichenko78} and the bibliography there.

Recent interest in Lie nilpotent associative algebras was motivated by the study of the quotients $L_{i}/L_{i+1}$ of the lower central series of the associated Lie algebra of an associative algebra $A$; here $L_n$ is the linear span in $A$ of the set of all commutators $[a_1, a_2, \dots , a_n]$ $(a_i \in A)$. The study of these quotients $L_i /L_{i+1}$ was initiated in 2007 in a pioneering article of Feigin and Shoikhet \cite{FS07} for $A = \mathbb C \langle X \rangle$; further results on this subject can be found, for example, in \cite{AE15,AJ10,BB11,BJ10,BEJKL12,CFZ15,DE08,DKM08,EKM09,JO13,Kerchev13}. Since $T^{(n)} (A)$ is the ideal in $A$ generated by $L_n$, some results about the quotients $T^{(i)} (A)/T^{(i+1)} (A)$ were obtained in these articles as well; in \cite{CFZ15,EKM09,JO13,Kerchev13} the latter quotients were the primary objects of study.

It can be easily seen that the ideal $T^{(2)}$ is generated (as a two-sided ideal in $K \langle X \rangle$) by the commutators $[x_1, x_2]$ $(x_i \in X)$. It is well-known that $T^{(3)}$ is generated by the polynomials
\[
[x_1,x_2,x_3], \qquad [x_1,x_2][x_3,x_4] + [x_1,x_3][x_2,x_4] \qquad (x_i \in X)
\]
(see, for instance, \cite{BEJKL12, Giambruno-Koshlukov01, GK99, Latyshev63}). If $\frac{1}{3} \in K$ then a similar generating set for $T^{(4)}$ contains the polynomials of 3 types (see \cite{EKM09, Gordienko07, Latyshev65, Volichenko78}):
\begin{equation}
\label{polynomial1}
[x_1,x_2,x_3,x_4] \qquad (x_i \in X),
\end{equation}
\begin{equation}
\label{polynomial2}
[x_1,x_2][x_3,x_4,x_5] \qquad (x_i \in X),
\end{equation}
\begin{equation}
\label{polynomial3}
\bigl( [x_1,x_2] [x_3,x_4] + [x_1,x_3] [x_2,x_4] \bigr) [x_5,x_6] \qquad (x_i \in X).
\end{equation}
If $\frac{1}{3} \notin K$ then a set of generators of $T^{(4)}$ consists of the polynomials of types (\ref{polynomial1}), (\ref{polynomial3}) and the polynomials
\begin{align*}
&   [x_1,x_2,x_3] [x_4,x_5,x_6] \qquad (x_i \in X),
\\
& [x_1,x_2,x_3] [x_4,x_5] + [x_1,x_2,x_4][x_3,x_5] \qquad (x_i \in X),
\\
& [x_1,x_2,x_3] [x_4,x_5] + [x_1,x_4,x_3][x_2,x_5] \qquad (x_i \in X)
\end{align*}
(see \cite[Theorem 1.3]{DK15}). In the latter case, in general, the polynomials $[x_1,x_2][x_3,x_4,x_5]$ $(x_i \in X)$ of type (\ref{polynomial2}) do not belong to $T^{(4)}$ but $3 \, [x_1,x_2][x_3,x_4,x_5] \in T^{(4)}$ (see \cite{Kras13}).

The aim of the present article is to exhibit a similar generating set for $T^{(5)}$. Our result is as follows.

\begin{theorem}
\label{TbaseT5}
Let $K$ be a unital associative and commutative ring and let $K \langle X \rangle$ be the free $K$-algebra on a non-empty set $X$ of free generators. Then the ideal $T^{(5)}$ is generated as a two-sided ideal of $K\left\langle X\right\rangle$ by the following polynomials:
\begin{equation}
\label{comm_x_5}
[x_{1} , x_{2} , x_{3}, x_{4} , x_{5} ] \qquad (x_i \in X);
\end{equation}
\begin{equation}
\label{prod_x_33}
[x_{1} , x_{2} , x_{3} ][x_{4} , x_{5} , x_{6} ] \qquad (x_i \in X);
\end{equation}
\begin{equation}
\label{prod_x_43}
[x_{1}, x_{2}, x_{3}, x_{4} ][x_{5}, x_{6}, x_{7} ] \qquad (x_i \in X);
\end{equation}
\begin{equation}
\label{prod_x_42}
[x_{1}, x_{2}, x_{3}, x_{4} ][x_{5}, x_{6} ] + [x_{1} , x_{2} , x_{3} , x_{5} ][x_{4}, x_{6} ] \qquad (x_i \in X);
\end{equation}
\begin{equation}
\label{prod_x_322}
[x_{1}, x_{2}, x_{3} ] \big( [x_{4}, x_{5} ][x_{6} , x_{7}] + [x_{4}, x_{6} ][x_{5}, x_{7} ] \big) \qquad (x_i \in X);
\end{equation}
\begin{equation}
\label{comm_x_2211}
\bigl[ \bigl( [x_{1},x_{2}][x_{3} , x_{4}] + [x_{1},x_{3}][x_{2},x_{4}] \bigr) , x_{5}, x_{6} \bigr] \qquad (x_i \in X);
\end{equation}
\begin{align}
\label{prod_x_2212}
& \bigl[ \big( [x_{1},x_{2}][x_{3},x_{4}] + [x_{1}, x_{3}] [x_{2}, x_{4}] \big) , x_{5} \bigr] [x_{6},x_{7}]
\\
+ & \bigl[ \big( [x_{1},x_{2}][x_{3},x_{4}] + [x_{1}, x_{3}] [x_{2},x_{4}] \big) ,x_{6} \bigr] [x_{5},x_{7}]  \qquad (x_i \in X); \nonumber
\end{align}
\begin{align}
\label{prod_x_2222}
& \bigl( [x_{1}, x_{2}][ x_{3}, x_{4}] + [x_{1} , x_{3}] [x_{2}, x_{4} ] \bigr)
\bigl( [x_{5}, x_{6} ][x_{7} , x_{8}]+ [x_{5}, x_{7} ][x_{6}, x_{8} ] \bigr)
\qquad (x_i \in X). 
\end{align}
\end{theorem}

%
%
%

It can be easily seen that the additive group of the ring $\mathbb Z\langle X \rangle / T^{(2)}$ is free abelian. It was shown in \cite{BEJKL12} that the additive group of $\mathbb Z \langle X \rangle / T^{(3)}$ is also free abelian. On the other hand, the additive group of the ring $\mathbb Z \langle X \rangle / T^{(4)}$ is a direct sum $A \oplus B$ of a free abelian group $A$ and an elementary abelian $3$-group $B$ (see \cite{Kras13}); bases of $A$ and $B$ were found in \cite{Kras13} and \cite{DK15}, respectively. Computational data by Cordwell, Fei and Zhou presented in \cite[Appendix A]{CFZ15} suggests that if $n \ge 6$ then the additive group of the ring $\mathbb Z \langle X \rangle / T^{(n)}$  is also a direct sum of a free abelian group and a non-trivial elementary abelian $3$-group while for $n = 5$ this group is free abelian. It is still an open problem whether the additive group of $\mathbb Z \langle X \rangle / T^{(5)}$ is free abelian indeed and, more generally, whether $K \langle X \rangle / T^{(5)}$ is a free $K$-module; we hope that our result might be the first step in solving it.


\section{Proof of Theorem  1.1 }

Let $A$ be an associative $K$-algebra. In the proof we will make use of the following identities:
\begin{align}
& [u,v_1 v_2] = v_1 [u, v_2] + [u, v_1] v_2; \qquad [u_1 u_2, v] = u_1 [u_2, v] + [u_1, v] u_2; \nonumber
\\
& [u_1 u_2, v, w] = u_1 [u_2, v, w] + [u_1, w] [u_2, v] + [u_1, v] [u_2, w] + [u_1, v, w] u_2; \label{comm_exp}
\\
& \begin{aligned} & [u_1 u_2, v, w, t] && = u_1 [u_2, v, w, t] + [u_1, t] [u_2, v, w] + [u_1, w] [u_2, v, t] + [u_1, w, t] [u_2, v]
\\
& &&  + [u_1, v] [u_2, w, t] + [u_1, v, t] [u_2, w] + [u_1, v, w] [u_2, t] + [u_1, v, w, t] u_2.
\end{aligned} \nonumber
\end{align}
These identities  hold for all $u_i, v_j, u, v, w, t \in A$; they can be checked by straightforward calculation or using the Leibniz rule since, for each $v \in A$, the map $u \rightarrow [u,v]$ is a derivation of $A$.

Let $I$ be the two-sided ideal of $K\left\langle X\right\rangle$ generated by the polynomials  (\ref{comm_x_5})--(\ref{prod_x_2222}). To prove the theorem one has to verify that $I \subseteq T^{(5)}$ and $T^{(5)} \subseteq I$.

To prove that $I \subseteq T^{(5)}$ we will check that all polynomials (\ref{comm_x_5})--(\ref{prod_x_2222}) belong to $T^{(5)}$. It is clear that the polynomials of the form (\ref{comm_x_5}) belong to $T^{(5)}$.

Further, for all $a_i \in K \langle X \rangle$ we have
\begin{align*}
& [a_1,a_2,a_3,a_4 a_5, a_6] = \bigl[ \bigl( a_4 [a_1,a_2,a_3,a_5] + [a_1,a_2,a_3,a_4] a_5 \bigr) , a_6 \bigr]
\\
= & a_4 [a_1,a_2,a_3,a_5,a_6] + [a_4,a_6][a_1,a_2,a_3, a_5] + [a_1,a_2,a_3,a_4][a_5,a_6]
\\
+ & [a_1,a_2,a_3,a_4,a_6]a_5 =  a_4 [a_1,a_2,a_3,a_5,a_6] + [a_1,a_2,a_3, a_5][a_4,a_6]
\\
+ & [a_1,a_2,a_3,a_4][a_5,a_6] - \bigl[ [a_1,a_2,a_3,a_4], [a_5,a_6] \bigr] +  [a_1,a_2,a_3,a_4,a_6]a_5 .
\end{align*}
Since $[a_1,a_2,a_3,a_4a_5, a_6], \, a_4 [a_1,a_2,a_3,a_5,a_6], \, [a_1,a_2,a_3,a_4,a_6]a_5 \in T^{(5)}$ and
\[
\bigl[ [a_1,a_2,a_3,a_4], [a_5,a_6] \bigr] = [a_1,a_2,a_3,a_4, a_5,a_6] - [a_1,a_2,a_3,a_4, a_6,a_5] \in T^{(5)},
\]
we have
\begin{equation}\label{42inT5}
[a_1,a_2,a_3,a_4][a_5,a_6] + [a_1,a_2,a_3, a_5][a_4,a_6] \in T^{(5)}
\end{equation}
for all $a_i \in K \langle X \rangle$. Thus, the polynomials of the form (\ref{prod_x_42}) belong to $T^{(5)}$.

To proceed further we need the following lemma.
\begin{lemma}\label{33inT5}
For all $a_i \in K \langle X \rangle$, we have $[a_1,a_2,a_3][a_4,a_5,a_6] \in T^{(5)}$.
\end{lemma}

\begin{proof}
Let $u_1, u_2$ be commutators of length 2 in $a_i$ $(a_i \in K \langle X \rangle )$ and let $y_1, y_2 \in K \langle X \rangle$. Then $[u_1, y_1 y_2, u_2] \in T^{(5)}$. We have
\begin{align*}
& [u_1, y_1 y_2 , u_2] = \bigl[ \bigl( y_1  [u_1, y_2] + [u_1, y_1] y_2 \bigr) , u_2 \bigr]
\\
= 
& y_1 [u_1, y_2, u_2] + [y_1, u_2] [u_1, y_2] + [u_1, y_1] [y_2, u_2] + [u_1, y_1, u_2] y_2.
\end{align*}
Since $[u_1, y_1 y_2, u_2], y_1 [u_1, y_2, u_2], [u_1, y_1, u_2] y_2 \in  T^{(5)}$, we have
\[
[y_1, u_2] [u_1, y_2] + [u_1, y_1] [y_2, u_2] \in  T^{(5)}.
\]

Let $u_1 = [a_1,a_2]$, $u_2 = [a_4, a_5]$, $y_1 = a_3$, $y_2= a_6$. Then
\begin{align*}
& [y_1, u_2] [u_1, y_2] + [u_1, y_1] [y_2, u_2]
= \bigl[ a_3, [a_4, a_5] \bigr] [a_1, a_2, a_6] + [a_1, a_2, a_3] \bigl[ a_6, [a_4, a_5] \bigr]
\end{align*}
so
\begin{equation}\label{(36)}
[a_1, a_2, a_3] [a_4, a_5, a_6] + [a_1, a_2, a_6] [ a_4, a_5, a_3]
\in T^{(5)} .
\end{equation}

Further, if $u_1 = [a_1,a_2]$ then $[u_1, a_3 a_4, a_5, a_6] \in T^{(5)}$. By (\ref{comm_exp}), we have
\begin{align*}
& [u_1, a_3 a_4, a_5, a_6] = \bigl[ \bigl( a_3[u_1, a_4] + [u_1, a_3]a_4 \bigr) , a_5, a_6 \bigr]
\\
= & a_3 [u_1, a_4, a_5, a_6] + [a_3, a_6][u_1, a_4, a_5] + [a_3, a_5] [u_1, a_4, a_6]
+ [a_3, a_5, a_6][u_1, a_4] 
\\
+  & [u_1, a_3] [a_4, a_5, a_6] + [u_1, a_3, a_6][a_4, a_5]
+ [u_1, a_3, a_5][ a_4, a_6] + [u_1, a_3, a_5, a_6] a_4 .
\end{align*}
It is clear that $[u_1, a_3 a_4, a_5, a_6], a_3 [u_1, a_4, a_5, a_6], [u_1, a_3, a_5, a_6] a_4 \in  T^{(5)}$. By (\ref{42inT5}), we have
\begin{gather*}
[a_3, a_6][u_1, a_4, a_5] + [a_3, a_5] [u_1, a_4, a_6]\in T^{(5)},
\\
[u_1, a_3, a_6][a_4, a_5] + [u_1, a_3, a_5][ a_4, a_6] \in  T^{(5)}.
\end{gather*}
It follows that $[a_3, a_5, a_6][u_1, a_4] + [u_1, a_3] [a_4, a_5, a_6] \in  T^{(5)}$, that is,
\begin{equation}\label{(34)}
[a_1, a_2, a_3] [a_4, a_5, a_6] + [a_1, a_2, a_4] [a_3, a_5, a_6]
\in  T^{(5)} .
\end{equation}

Let $a_i \in K \langle X \rangle$ $( 1 \le i \le 6)$. It
is clear that
\[
[a_1, a_2, a_3] [a_4, a_5, a_6] = \mbox{sgn } \sigma [a_{\sigma
(1)}, a_{\sigma (2)}, a_{\sigma (3)}] [a_{\sigma (4)}, a_{\sigma
(5)}, a_{\sigma (6)}]
\]
if $\sigma = (12)$ or $\sigma = (45)$. By (\ref{(36)}) and
(\ref{(34)}),
\begin{align}\label{alternating}
& [a_1, a_2, a_3] [a_4, a_5, a_6]
\equiv \mbox{sgn } \sigma \ [a_{\sigma (1)}, a_{\sigma (2)}, a_{\sigma (3)}] [a_{\sigma (4)}, a_{\sigma (5)}, a_{\sigma (6)}] \pmod{T^{(5)}} 
\end{align}
if $\sigma \in \{ (36), (34), (16) \}$. It is easy to check that the transpositions $(12)$, $(45)$, $(36)$, $(34)$ and $(16)$ generate the entire group $S_6$ of the permutations on the set $\{ 1, 2, 3, 4, 5, 6 \}$. It follows that the congruence (\ref{alternating}) holds for all $\sigma \in S_6$ and all $a_i \in K \langle X \rangle$ $( 1 \le i \le 6)$.

Now we are in a position to complete the proof of the lemma. By the Jacobi identity, we have
\[
([a_1, a_2, a_3] + [a_2, a_3, a_1] + [a_3, a_1, a_2])[a_4,a_5,a_6] = 0
\]
By (\ref{alternating}), this implies
\begin{equation}\label{3}
3 \ [a_1, a_2, a_3][a_4,a_5,a_6] \in T^{(5)}.
\end{equation}
Further, it is clear that
\[
[a_1, a_2, a_3][a_4,a_5,a_6] \equiv [a_4,a_5,a_6][a_1, a_2, a_3]
\pmod{T^{(5)}} .
\]
On the other hand, by (\ref{alternating}) we have
\[
[a_1, a_2, a_3][a_4,a_5,a_6] \equiv - [a_4,a_5,a_6] [a_1, a_2,
a_3] \pmod{T^{(5)}}.
\]
It follows that
\begin{equation}\label{2}
2 \ [a_1, a_2, a_3][a_4,a_5,a_6] \in T^{(5)}.
\end{equation}
Now (\ref{3}) and (\ref{2}) imply that $[a_1, a_2, a_3][a_4,a_5,a_6] \in T^{(5)}$ for all $a_i \in K \langle
X \rangle$ $(1 \le i \le 6)$, as required. The proof of Lemma \ref{33inT5} is completed.
\end{proof}

It follows immediately from Lemma \ref{33inT5} that the polynomials of the forms (\ref{prod_x_33}) and (\ref{prod_x_43}) belong to $T^{(5)}$.

Further, by (\ref{comm_exp}) we have
\begin{align*}
& [a_5 a_6, a_4, a_7] = a_5 [a_6, a_4, a_7] + [a_5,a_4][a_6,a_7] + [a_5,a_7][a_6,a_4]
+ [a_5,a_4,a_7]a_6 
\\
= & [a_6, a_4, a_7]a_5 - [a_6,a_4,a_7,a_5] - [a_4,a_5][a_6,a_7]
- [a_4,a_6] [a_5,a_7] +  \bigl[ [a_5,a_7],[a_6,a_4] \bigr] + [a_5,a_4,a_7]a_6
\end{align*}
for all $a_i \in K \langle X \rangle$. Hence,
\begin{align*}
& [a_1,a_2,a_3] [a_5 a_6,a_4,a_7] = [a_1,a_2,a_3] [a_6, a_4, a_7] a_5
- [a_1,a_2,a_3] [a_6,a_4,a_7,a_5] 
\\
- &[a_1,a_2,a_3] \bigl( [a_4,a_5][a_6,a_7] + [a_4,a_6] [a_5,a_7] \bigr)
+[a_1,a_2,a_3] \bigl[ [a_5,a_7],[a_6,a_4] \bigr] + [a_1,a_2,a_3] [a_5,a_4,a_7]a_6 .
\end{align*}
By Lemma \ref{33inT5}, the polynomials
\begin{gather*}
[a_1,a_2,a_3] [a_5 a_6,a_4,a_7], \quad [a_1,a_2,a_3] [a_6, a_4, a_7] a_5, \quad  [a_1,a_2,a_3] [a_6,a_4,a_7,a_5],
\\
[a_1,a_2,a_3] \bigl[ [a_5,a_7],[a_6,a_4], \quad [a_1,a_2,a_3] [a_5,a_4,a_7]a_6
\end{gather*}
belong to $T^{(5)}$. It follows that for all $a_i \in K \langle X \rangle$
\begin{equation}
\label{322inT5}
[a_1,a_2,a_3] \bigl( [a_4,a_5][a_6,a_7] + [a_4,a_6] [a_5,a_7] \bigr) \in T^{(5)} .
\end{equation}
Thus, the polynomials of the form (\ref{prod_x_322}) belong to $T^{(5)}$.

Similarly, Lemma \ref{33inT5}, (\ref{322inT5}) and the equality
\begin{align*}
& [a_2 a_3, a_1, a_4] \bigl( [a_5,a_6][a_7,a_8] + [a_5,a_7] [a_6,a_8] \bigr)
\\
= & a_2 [a_3, a_1, a_4]  \bigl( [a_5,a_6][a_7,a_8] + [a_5,a_7] [a_6,a_8] \bigr)
\\
- & \bigl( [a_1,a_2][a_3,a_4] + [a_1,a_3] [a_2,a_4] \bigr)  \bigl( [a_5,a_6][a_7,a_8] + [a_5,a_7] [a_6,a_8] \bigr)
\\
+ & \bigl[ [a_2,a_4],[a_3,a_1] \bigr]  \bigl( [a_5,a_6][a_7,a_8] + [a_5,a_7] [a_6,a_8] \bigr)
\\
+ & \bigl( a_3 [a_2,a_1,a_4] +  [a_2,a_1,a_4, a_3]\bigr)  \bigl( [a_5,a_6][a_7,a_8] + [a_5,a_7] [a_6,a_8] \bigr)
\end{align*}
imply that for all $a_i \in K \langle X \rangle$
\[
\bigl( [a_1,a_2][a_3,a_4] + [a_1,a_3] [a_2,a_4] \bigr)  \bigl( [a_5,a_6][a_7,a_8] + [a_5,a_7] [a_6,a_8] \bigr) \in T^{(5)} .
\]
It follows that the polynomials of the form (\ref{prod_x_2222}) belong to $T^{(5)}$.

One can check in a similar way using the equalities
\begin{align*}
[a_2a_3, a_1, a_4, a_5, a_6] = & \Big[ \Big( a_2 [a_3, a_1, a_4]  - \bigl( [a_1,a_2][a_3,a_4] + [a_1,a_3] [a_2,a_4] \bigr)
\\
+ & \bigl[ [a_2,a_4],[a_3,a_1] \bigr]  +  [a_2,a_1,a_4] a_3 \Big)  , a_5,  a_6 \Big]
\end{align*}
and
\begin{align*}
& [a_2a_3, a_1, a_4, a_5] [a_6, a_7] + [a_2a_3, a_1, a_4, a_6] [a_5, a_7]
\\
= & \Big[ \Big( a_2 [a_3, a_1, a_4]  - \bigl( [a_1,a_2][a_3,a_4] + [a_1,a_3] [a_2,a_4] \bigr)
+ \bigl[ [a_2,a_4],[a_3,a_1] \bigr]  +  [a_2,a_1,a_4] a_3 \Big)  , a_5 \Big] [ a_6, a_7]
\\
+ & \Big[ \Big( a_2 [a_3, a_1, a_4]  - \bigl( [a_1,a_2][a_3,a_4] + [a_1,a_3] [a_2,a_4] \bigr)
+ \bigl[ [a_2,a_4],[a_3,a_1] \bigr]  +  [a_2,a_1,a_4] a_3 \Big)  , a_6 \Big] [ a_5, a_7] ,
\end{align*}
Lemma \ref{33inT5} and (\ref{42inT5}) that the polynomials of the forms (\ref{comm_x_2211}) and (\ref{prod_x_2212}) belong to $T^{(5)}$ as well.

Thus, all polynomials of the forms (\ref{comm_x_5})--(\ref{prod_x_2222}) are contained in $T^{(5)}$. It follows that $I \subseteq T^{(5)}$.


\bigskip
Now we prove that $T^{(5)} \subseteq I$. Since the ideal $T^{(5)}$ is generated by the polynomials $[a_{1} , a_{2} , a_{3}, a_{4} , a_{5} ]$ where $a_i \in K\left\langle X\right\rangle$ for all $i$, it suffices to check that, for all $a_i \in K\left\langle X\right\rangle$, the polynomial
\begin{equation}
\label{comm_a_5}
[a_{1} , a_{2} , a_{3}, a_{4} , a_{5}]
\end{equation}
belongs to $I$. Clearly, one may assume without loss of generality that all $a_i$ in (\ref{comm_a_5}) are monomials.

In order to prove that each polynomial of the form (\ref{comm_a_5}) belongs to $I$ we will prove that, for all monomials $a_i \in K\left\langle X\right\rangle$, the following polynomials belong to $I$ as well:
\begin{equation}
\label{prod_a_2222}
\bigl([a_1 , a_2][a_3 , a_4 ]+[a_1 , a_3][a_2 , a_4 ]\bigr) \bigl([a_5,a_6][a_7 , a_8]+[a_5 , a_7][a_6 , a_8 ]\bigr) ;
\end{equation}
\begin{equation}
\label{prod_a_43}
[a_{1}, a_{2}, a_{3}, a_{4} ][a_{5}, a_{6}, a_{7} ] ;
\end{equation}
\begin{equation}
\label{prod_a_322}
[a_{1} , a_{2} , a_{3} ]\bigl([a_4,a_5][a_{6} , a_7]+[a_{4} , a_6][a_5 , a_{7} ]\bigr) ;
\end{equation}
\begin{equation}
\label{prod_a_2212}
\Big[ \big( [a_{1},a_{2}][a_{3},a_{4}]+[a_{1},a_{3}][a_{2},a_{4}] \big) ,a_{5} \Big] [a_{6},a_{7}]
\end{equation}
\[
+ \Big[ \big( [a_{1},a_{2}][a_{3},a_{4}]+[a_{1},a_{3}][a_{2},a_{4}] \big) ,a_{6} \Big] [a_{5},a_{7}] ;
\]
\begin{equation}
\label{comm_a_2211}
\Big[ \big( [a_{1},a_{2}][a_{3},a_{4}]+[a_{1},a_{3}][a_{2},a_{4}] \big) , a_{5},a_{6} \Big] ;
\end{equation}
\begin{equation}
\label{prod_a_33}
[a_{1} , a_{2} , a_{3} ][a_{4} , a_{5} , a_{6} ] ;
\end{equation}
\begin{equation}
\label{prod_a_42}
[a_{1}, a_{2}, a_{3}, a_{4} ][a_{5}, a_6 ] + [a_{1} , a_{2} , a_{3} , a_{5} ][a_{4}, a_6 ] .
\end{equation}
The proof is by induction on the degree $m = \deg f$ of a polynomial $f$ that is of one of the forms (\ref{comm_a_5})--(\ref{prod_a_42}). It is clear that for such $f$ we have $m  \ge 5$. If $m = 5$ then $f $ is of the form (\ref{comm_a_5}) with all monomials $a_i$ of degree $1$. Hence, $f$ is of the form (\ref{comm_x_5}) so $f \in I$. This establishes the base of the induction.

To prove the induction step suppose that $m = \deg f >5$ and that all polynomials of the forms (\ref{comm_a_5})--(\ref{prod_a_42}) of degree less than $m$ belong to $I$.

We need the following.

\begin{lemma}
\label{lemma_422}
Let $b_1, b_2, \dots ,b_8 \in K \langle X \rangle$ be monomials such that $\sum_{i=1}^8 \deg b_i = m$. Suppose that all polynomials of the forms (\ref{comm_a_5})--(\ref{prod_a_42}) of degree less than $m$ belong to $I$. Then
\[
[b_1,b_2,b_3,b_4] \bigl([b_5,b_6][b_7 , b_8] + [b_5 , b_7][b_6 , b_8 ]\bigr) \in I.
\]
\end{lemma}

\begin{proof}
We have
\begin{align}
\label{case1_234_422}
& [b_1,b_2,b_3,b_4] \bigl([b_5,b_6][b_7 , b_8] + [b_5 , b_7][b_6 , b_8 ]\bigr)
=  [b_1,b_2,b_3] \bigl([b_5,b_6][b_7 , b_8] + [b_5 , b_7][b_6 , b_8 ]\bigr) b_4
\\
- &
b_4 [b_1,b_2,b_3] \bigl([b_5,b_6][b_7 , b_8] + [b_5 , b_7][b_6 , b_8 ]\bigr)
-
[b_1,b_2,b_3] \Big[ \bigl([b_5,b_6][b_7 , b_8] + [b_5 , b_7][b_6 , b_8 ]\bigr) ,b_4 \Big] .
\nonumber
\end{align}
Here
\[
[b_1,b_2,b_3] \bigl([b_5,b_6][b_7 , b_8] + [b_5 , b_7][b_6 , b_8 ]\bigr) \in I
\]
because it is a polynomial of type (\ref{prod_a_322}) of degree less than $m$. It follows that
\begin{equation}
\label{case1_234_422_1}
 [b_1,b_2,b_3] \bigl([b_5,b_6][b_7 , b_8] + [b_5 , b_7][b_6 , b_8 ]\bigr) b_4 - 
b_4 [b_1,b_2,b_3] \bigl([b_5,b_6][b_7 , b_8] + [b_5 , b_7][b_6 , b_8 ]\bigr) \in I.
\end{equation}
On the other hand, 
\begin{align*}
 \bigl[ [b_5,b_6][b_7 , b_8]  ,b_4 \bigr]  =  & [b_5,b_6, b_4][b_7 , b_8] +  [b_5,b_6][b_7 , b_8 ,b_4 ]  
 \\
 =  & [b_5,b_6, b_4][b_7 , b_8] + [b_7 , b_8 ,b_4 ] [b_5,b_6] - \bigl[ [b_7 , b_8 ,b_4 ], [b_5,b_6] \bigr]
 \\
\equiv &  [b_5,b_6, b_4][b_7 , b_8] + [b_7 , b_8 ,b_4 ] [b_5,b_6] \pmod{I}
\end{align*}
because 
\[
\bigl[ [b_7 , b_8 ,b_4 ], [b_5,b_6] \bigr] = [b_7 , b_8 ,b_4 , b_5,b_6]  - [b_7 , b_8 ,b_4 , b_6,b_5] 
\]
and $ [b_7 , b_8 ,b_4 , b_5,b_6], [b_7 , b_8 ,b_4 , b_6,b_5] \in I$ since these two polynomials are of type (\ref{comm_a_5}) of degree less than $m$. Similarly, 
\[
 \bigl[ [b_5,b_7][b_6 , b_8]  ,b_4 \bigr] \equiv  [b_5,b_7, b_4][b_6 , b_8] + [b_6 , b_8 ,b_4 ] [b_5,b_7] \pmod{I} .
\]
It follows that
\begin{align}
\label{case1_234_422_21}
& [b_1,b_2,b_3] \Big[ \bigl([b_5,b_6][b_7 , b_8] + [b_5 , b_7][b_6 , b_8 ]\bigr) ,b_4 \Big] 
\equiv [b_1,b_2,b_3] \bigl(  [b_5,b_6, b_4][b_7 , b_8] 
\\
+ & [b_7 , b_8 ,b_4 ] [b_5,b_6]  + [b_5,b_7, b_4][b_6 , b_8] + [b_6 , b_8 ,b_4 ] [b_5,b_7] \bigr) \pmod{I} .
\nonumber
\end{align}
Note that if $k, \ell \in \{5,6,7,8\}$, $k \ne \ell$ then
\begin{equation}
\label{case1_234_422_22}
[b_1,b_2,a_3]  [b_{k},b_{\ell}, b_4] \in I
\end{equation}
because it is a polynomial of type (\ref{prod_a_33}) of degree less than $m$. It follows from (\ref{case1_234_422_21}) and (\ref{case1_234_422_22}) that 
\begin{equation}
\label{case1_234_422_2}
[b_1,b_2,b_3] \Big[ \bigl([b_5,b_6][b_7 , b_8] + [b_5 , b_7][b_6 , b_8 ]\bigr) ,b_4 \Big]  \in I
\end{equation}
and, therefore, by (\ref{case1_234_422}), (\ref{case1_234_422_1}) and (\ref{case1_234_422_2}),
\[
[b_1,b_2,b_3,b_4] \bigl([b_5,b_6][b_7 , b_8] + [b_5 , b_7][b_6 , b_8 ]\bigr) \in I,
\]
as required.
\end{proof}

Since $\bigl[ [b_1,b_2],[b_3,b_4] \bigr] = [b_1,b_2,b_3,b_4] - [b_1,b_2,b_4,b_3]$,
we have the following:
\begin{corollary}
\label{corollary_422}
Suppose that $b_1, b_2, \dots ,b_8 $ are as in Lemma \ref{lemma_422}. Then
\[
\bigl[ [b_1,b_2],[b_3,b_4] \bigr] \bigl([b_5,b_6][b_7 , b_8] + [b_5 , b_7][b_6 , b_8 ]\bigr) \in I.
\]
\end{corollary}

It is clear that the following lemma holds.
\begin{lemma}
\label{comm-comm}
Let $b_1, b_2, b_3,b_4 ,b_5 \in K \langle X \rangle$ be monomials such that $\sum_{i=1}^5 \deg b_i < m$. Suppose that all polynomials of the forms (\ref{comm_a_5})--(\ref{prod_a_42}) of degree less than $m$ belong to $I$. Then
\[
b_1[b_2,b_3,b_4,b_5] \equiv [b_2,b_3,b_4,b_5] b_1 \pmod{I}  \quad \mbox{\rm and} \quad [b_1,b_2,b_3][b_4,b_5] \equiv [b_4,b_5] [b_1, b_2, b_3] \pmod{I} .
\]
\end{lemma}

\medskip
Now we prove the induction step. We split the proof into 8 cases.


\textbf{Case 1.} Suppose that $f$ is of the form (\ref{prod_a_2222}),
\begin{align*}
f & = f(a_1, a_2, \dots , a_8) = \bigl([a_1 , a_2] [a_3 , a_4 ] + [a_1 , a_3][a_2 , a_4 ]\bigr) \bigl([a_5,a_6][a_7 , a_8] + [a_5 , a_7][a_6 , a_8 ]\bigr) .
\end{align*}
If $\deg f = 8$ then each  monomial $a_i$ is of degree $1$ so $f$ is of the form (\ref{prod_x_2222}) and, therefore, $f \in I$. If $\deg f > 8$ then, for some $i$, $1\le i \le 8$, we have $a_i=a_i' a_i''$ where $\deg a_i'$, $\deg a_i''< \deg a_i$. We claim that to check that $f \in I$ one may assume without loss of generality that $i = 1 $, that is, $a_1 = a_1' a_1''$.

Indeed,  by the induction hypothesis,
\[
\Big[ \bigl([a_1 , a_2] [a_3 , a_4 ] + [a_1 , a_3][a_2 , a_4 ]\bigr) , a_k, a_{\ell} \Big] \in I
\]
if $5 \le k , \ell \le 8$, $k \ne \ell $, because the commutator above is a polynomial of type (\ref{comm_a_2211}) and of degree less than $m = \deg f$. Hence,
\begin{align*}
 & \Big[ \bigl([a_1 , a_2] [a_3 , a_4 ] + [a_1 , a_3][a_2 , a_4 ]\bigr) , [a_k, a_{\ell}] \Big]
\\
= & \Big[ \bigl([a_1 , a_2] [a_3 , a_4 ] + [a_1 , a_3][a_2 , a_4 ]\bigr) , a_k, a_{\ell} \Big]
- \Big[ \bigl([a_1 , a_2] [a_3 , a_4 ] + [a_1 , a_3][a_2 , a_4 ]\bigr) , a_{\ell}, a_k \Big] \in I .
\end{align*}
It follows that 
\begin{align*}
& \bigl([a_1 , a_2] [a_3 , a_4 ] + [a_1 , a_3][a_2 , a_4 ]\bigr) \bigl([a_5,a_6][a_7 , a_8] + [a_5 , a_7][a_6 , a_8 ]\bigr)
\\
\equiv & \bigl([a_5,a_6][a_7 , a_8] + [a_5 , a_7][a_6 , a_8 ]\bigr) \bigl([a_1 , a_2] [a_3 , a_4 ] + [a_1 , a_3][a_2 , a_4 ]\bigr) \pmod{I},
\end{align*}
that is,
\begin{equation}
\label{case1_5678}
f(a_1,a_2,a_3,a_4, a_5,a_6,a_7,a_8) \equiv f(a_5,a_6,a_7,a_8,a_1,a_2,a_3,a_4) \pmod{I} .
\end{equation}

Further, let $g = \bigl([a_5,a_6][a_7 , a_8] + [a_5 , a_7][a_6 , a_8 ]\bigr) $; then, by Corollary \ref{corollary_422}, we have
\begin{equation*}
\label{case1_234_22g}
\bigl[ [a_1,a_2], [a_3, a_4] \bigr] g \equiv 0 \pmod{I}.
\end{equation*}
It follows that
\begin{align*}
& \bigl( [a_1,a_2] [a_3,a_4] + [a_1,a_3][a_2,a_4] \bigr) g \equiv \bigl( [a_2,a_1] [a_4,a_3] + [a_2,a_4][a_1,a_3] \bigr) g \pmod{I}
\\
\equiv & \bigl( [a_3,a_4] [a_1,a_2] + [a_3,a_1][a_4,a_2] \bigr) g \pmod{I} \equiv    \bigl( [a_4,a_3] [a_2,a_1] + [a_4,a_2][a_3,a_1] \bigr) g \pmod{I} ,
\end{align*}
that is, 
\begin{align}
\label{case1_234}
& f(a_1,a_2,a_3,a_4,a_5, \dots ) \equiv f(a_2,a_1,a_4,a_3,a_5, \dots ) \pmod{I} 
\\
\equiv & f(a_3,a_4,a_1,a_2,a_5, \dots ) \pmod{I} \equiv f(a_4,a_3,a_2,a_1, a_5, \dots ) \pmod{I}.
\nonumber
\end{align}

It follows from (\ref{case1_5678}) and (\ref{case1_234}) that if $a_i = a_i' a_i''$ for some $i$, $1 \le i \le 8$, then we may assume without loss of generality that $i=1$, that is, $a_1 = a_1' a_1''$, as claimed. 

Further, we have
\[
[a_1'a_1'', a_i] = a_1'[a_1'', a_i] + [a_1', a_i] a_1'' = a_1' [a_1'', a_i] + a_1'' [a_1', a_i] + [a_1', a_i, a_1'']
\]
where $i \in \{ 2,3\}$ so
\begin{align}
\label{case1}
f & = \big( [a_1'a_1'', a_2][a_3 , a_4 ] + [a_1'a_1'' , a_3][a_2 , a_4 ] \big) g
= a_1' \big( [a_1'', a_2][a_3 , a_4 ] + [a_1'', a_3][a_2 , a_4 ] \big) g
\\
& + a_1'' \big( [a_1', a_2][a_3 , a_4 ] + [a_1', a_3][a_2 , a_4 ] \big) g
+ \big( [a_1', a_2,a_1''][a_3 , a_4 ] + [a_1', a_3,a_1''][a_2 , a_4 ] \big) g .
\nonumber
\end{align}
By the induction hypothesis,
\[
\big( [b, a_2][a_3 , a_4 ] + [b, a_3][a_2 , a_4 ] \big) g =
\big( [b, a_2][a_3 , a_4 ] + [b, a_3][a_2 , a_4 ] \big) \big( [a_5,a_6][a_7 , a_8] + [a_5 , a_7][a_6 , a_8 ] \big) \in I
\]
if $b \in \{ a_1',a_1'' \}$ so
\begin{align}
\label{case1_1}
& a_1' \big( [a_1'', a_2][a_3 , a_4 ] + [a_1'', a_3][a_2 , a_4 ] \big) g 
+ a_1'' \big( [a_1', a_2][a_3 , a_4 ] + [a_1', a_3][a_2 , a_4 ] \big) g
\in I.
\end{align}
On the other hand, by Lemma \ref{comm-comm},
\begin{equation}
\label{case1_32-23}
\big( [a_1', a_2,a_1''][a_3 , a_4 ] + [a_1', a_3,a_1''][a_2 , a_4 ] \big) g \equiv \big( [a_3 , a_4 ] [a_1', a_2,a_1''] + [a_2 , a_4 ] [a_1', a_3,a_1''] \big) g \pmod{I} .
\end{equation}
By the induction hypothesis, we have
\[
 [a_1', a_2,a_1''] g =  [a_1', a_2,a_1''] \big( [a_5,a_6][a_7 , a_8]+[a_5 , a_7][a_6 , a_8 ] \big) \in I
\]
because this is a polynomial of type (\ref{prod_a_322}) of degree less than $m$. Hence, $[a_3 , a_4 ] [a_1', a_2,a_1''] g \in I$ and, similarly,  $[a_2 , a_4 ] [a_1', a_3,a_1''] g \in I$. By (\ref{case1_32-23}), we have
\begin{equation}
\label{case1_2}
\big( [a_1', a_2,a_1''][a_3 , a_4 ] + [a_1', a_3,a_1''][a_2 , a_4 ] \big) g \in I
\end{equation}
so, by  (\ref{case1}), (\ref{case1_1}) and (\ref{case1_2}), $f \in I$.

Thus, each polynomial
\[
f = \big( [a_1 , a_2] [a_3 , a_4 ] + [a_1 , a_3][a_2 , a_4 ] \big) \big( [a_5,a_6][a_7 , a_8] + [a_5 , a_7][a_6 , a_8 ] \big)
\]
of the form (\ref{prod_a_2222}) of degree $m$ belongs to $I$.


\medskip
\textbf{Case 2.} Suppose now that $f$ is of the form (\ref{prod_a_43}), 
\[
f = [a_1,a_2,a_3,a_4] [a_5,a_6,a_7].
\] 
If $\deg f = 7$ then $f$ is of the form (\ref{prod_x_43}) so $f \in I$. If $\deg f>7$ then for some $i$, $1\le i \le 7$, we have $a_i=a_i'a_i''$, where $\deg a_i'$, $\deg a_i'' < \deg a_i$. 

Suppose that $a_1 = a_1' a_1''$. By (\ref{comm_exp}), we have
\begin{align*}
\label{case4_1_1}
f = & [a_1' a_1'',a_2,a_3,a_4] [a_5,a_6,a_7] 
\\
= & \bigl( a_1' [a_1'',a_2,a_3,a_4] + [a_1',a_2] [a_1'', a_3,a_4] + [a_1', a_3] [a_1'', a_2,a_4] 
+   [a_1', a_4][a_1'',a_2,a_3]  
\nonumber
\\
+ & [a_1',a_2,a_3][a_1'',a_4]  +  [a_1',a_2,a_4][a_1'',a_3] + [a_1',a_3,a_4][a_1'',a_2] + [a_1',a_2,a_3,a_4]a_1'' \bigr) [a_5,a_6,a_7]
\nonumber
\end{align*}
so, by Lemma \ref{comm-comm},
\begin{align}
f \equiv & \bigl( a_1' [a_1'',a_2,a_3,a_4] + [a_1',a_2] [a_1'', a_3,a_4] + [a_1', a_3] [a_1'', a_2,a_4] 
+   [a_1', a_4][a_1'',a_2,a_3]  + [a_1'',a_4][a_1',a_2,a_3]
\\
  + & [a_1'',a_3] [a_1',a_2,a_4] + [a_1'',a_2][a_1',a_3,a_4] + a_1''[a_1',a_2,a_3,a_4]  \bigr) [a_5,a_6,a_7]  \pmod{I}.
\nonumber
\end{align}
By the induction hypothesis, we have $[b,a_2,a_3,a_4][a_5,a_6,a_7] \in I$ if $b \in \{ a_1', a_1'' \}$ so 
\begin{equation}
\label{case4_1_11}
\bigl( a_1' [a_1'',a_2,a_3,a_4] + a_1''[a_1',a_2,a_3,a_4] \bigr) [a_5,a_6,a_7] \in I.
\end{equation}
Further, by the induction hypothesis, $[b,a_k,a_{\ell}][a_5,a_6,a_7] \in I$ if $b \in \{ a_1', a_1''\}$ and $2 \le k, \ell \le 4$, $k \ne \ell$, so
\begin{align}
\label{case4_1_12}
& \bigl( [a_1',a_2] [a_1'', a_3,a_4] + [a_1', a_3] [a_1'', a_2,a_4] 
+   [a_1', a_4][a_1'',a_2,a_3]  +  [a_1'',a_4][a_1',a_2,a_3] 
\\
 & +   [a_1'',a_3] [a_1',a_2,a_4] + [a_1'',a_2][a_1',a_3,a_4] \bigr) [a_5,a_6,a_7] \in I
\nonumber
\end{align}
It follows from (\ref{case4_1_1}), (\ref{case4_1_11}) and  (\ref{case4_1_12})  that $f \in I$ if $a_1 = a_1'a_1''$. Similarly, $f \in I$ if $a_2 = a_2'a_2''$.

Suppose that $a_3 = a_3' a_3''$. We have
\begin{align*}
\label{case4_3_1}
f = & [a_1,a_2,a_3'a_3'',a_4] [a_5,a_6,a_7] = \bigl[ a_3' [a_1,a_2,a_3''] +[a_1,a_2,a_3']a_3'' ,a_4 \bigr] [a_5,a_6,a_7]
\\
= & \bigl( a_3'[a_1,a_2,a_3'',a_4] + [a_3',a_4][a_1,a_2,a_3''] + [a_1,a_2,a_3'] [a_3'',a_4]  + [a_1,a_2,a_3',a_4] a_3''   \bigr) [a_5,a_6,a_7]
\nonumber
\\
\equiv & \bigl( a_3'[a_1,a_2,a_3'',a_4] + [a_3',a_4][a_1,a_2,a_3'']  +  [a_3'',a_4][a_1,a_2,a_3']  + a_3''[a_1,a_2,a_3',a_4]  \bigr) [a_5,a_6,a_7] \pmod{I} .
\end{align*}
By the induction hypothesis, $[a_1,a_2,b,a_4] [a_5,a_6,a_7] \in I$ and $[a_1,a_2,b] 
[a_5,a_6,a_7] \in I$ if $b \in \{ a_3', a_3'' \}$ so
\[
\bigl( a_3'[a_1,a_2,a_3'',a_4] + [a_3',a_4][a_1,a_2,a_3''] + a_3''[a_1,a_2,a_3',a_4] +  
[a_3'',a_4][a_1,a_2,a_3'] \bigr) [a_5,a_6,a_7] \in I. 
\]
It follows that $f \in I$ if $a_3 = a_3' a_3''$.

Suppose that $a_4 = a_4'a_4''$. Then
\begin{align*}
f = & [a_1,a_2,a_3,a_4' a_4''] [a_5,a_6,a_7] = \bigl( a_4'[a_1,a_2,a_3,a_4''] + [a_1,a_2,a_3,a_4']  a_4''  \bigr) [a_5,a_6,a_7]
\\
\equiv & \bigl( a_4'[a_1,a_2,a_3,a_4''] + a_4'' [a_1,a_2,a_3,a_4'] \bigr) [a_5,a_6,a_7]  \pmod{I}.
\end{align*}
By the induction hypothesis,  $[a_1,a_2,a_3,b] [a_5,a_6,a_7] \in I$ if $b \in \{ a_4', a_4'' \}$. It follows that $f \in I$ if $a_4 = a_4'a_4''$.

Now suppose that $a_5 = a_5'a_5''$. Then
\begin{align*}
f = & [a_1,a_2,a_3,a_4] [a_5' a_5'',a_6,a_7]  
\\
= & [a_1,a_2,a_3,a_4] \bigl( a_5' [a_5'',a_6,a_7] +   [a_5',a_7] [a_5'',a_6] + [a_5',a_6][ a_5'',a_7] + [a_5',a_6,a_7]a_5'' \bigr)
\\
\equiv & a_5' [a_1,a_2,a_3,a_4]  [a_5'',a_6,a_7] +  [a_1,a_2,a_3,a_4] \bigl( [a_5',a_7] [a_5'',a_6] + [a_5',a_6][ a_5'',a_7] \bigr)
\\
 + & [a_1,a_2,a_3,a_4] [a_5',a_6,a_7]a_5'' \pmod{I} .
\end{align*}
Here $a_5' [a_1,a_2,a_3,a_4]  [a_5'',a_6,a_7] \in I$ and $[a_1,a_2,a_3,a_4] [a_5',a_6,a_7]a_5'' \in I$ by the induction hypothesis and 
\[
 [a_1,a_2,a_3,a_4] \bigl( [a_5',a_7] [a_5'',a_6] + [a_5',a_6][ a_5'',a_7] \bigr) = - [a_1,a_2,a_3,a_4] \bigl( [a_5',a_7] [a_6, a_5''] + [a_5',a_6][ a_7, a_5''] \bigr) \in I
\]
by Lemma \ref{lemma_422}. Therefore, $f \in I$ if $a_5 = a_5'a_5''$. Similarly, $f \in I$ if $a_6 = a_6'a_6''$. 

Finally, suppose that $a_7=a_7'a_7''$. We have
\begin{align*}
f = & [a_1,a_2,a_3,a_4] [a_5,a_6,a_7' a_7'']  = [a_1,a_2,a_3,a_4] a_7' [a_5,a_6, a_7'']  +[a_1,a_2,a_3,a_4] [a_5,a_6,a_7' ] a_7''
\\
\equiv & a_7' [a_1,a_2,a_3,a_4]  [a_5,a_6, a_7'']  + [ a_1,a_2,a_3,a_4] [a_5,a_6,a_7' ] a_7'' \pmod{I}.
\end{align*}
Here $a_7' [a_1,a_2,a_3,a_4]  [a_5,a_6, a_7''] \in I$ and $[ a_1,a_2,a_3,a_4] [a_5,a_6,a_7' ] a_7'' \in I$ by the induction hypothesis so $f \in I$ if $a_7 = a_7' a_7''$.

Thus, each polynomial $f$ of the form  (\ref{prod_a_43}) and of degree $m$ belongs to $I$.


\medskip
\textbf{Case 3.} Suppose that $f$ is of the form (\ref{prod_a_322}),
\[
f = f(a_1, \dots ,a_7) = [a_{1} , a_{2} , a_{3} ]\bigl([a_4,a_5][a_{6} , a_7]+[a_{4} , a_6][a_5 , a_{7} ]\bigr) .
\]
If $\deg f=7$  then each monomial $a_i$ is of degree $1$ so $f$ is of the form (\ref{prod_x_322}) and, therefore, $f \in I$. If $\deg f>7$ then for some $i$, $1\le i \le 7$, we have $a_i=a_i'a_i''$, where $\deg a_i'$, $\deg a_i'' < \deg a_i$. We claim that to check that $f \in I$ one may assume without loss of generality that $i \in \{ 1, 3,4 \} $.

Indeed, it is clear that 
\begin{equation}
\label{case2_1234}
f(a_1,a_2,a_3, \dots ,a_7) = - f(a_2,a_1,a_3, \dots ,a_7).
\end{equation} 
On the other hand, 
\begin{align*}
& [a_{1} , a_{2} , a_{3} ] \bigl[ [a_4,a_5], [a_{6} , a_7] \bigr]  =  [a_{1} , a_{2} , a_{3} ] \bigl( [a_4,a_5, a_{6} , a_7] - [a_4,a_5, a_7 , a_6] \bigr) 
\\
\equiv & \bigl( [a_4,a_5, a_{6} , a_7] - [a_4,a_5, a_7 , a_6] \bigr) [a_{1} , a_{2} , a_{3} ] \pmod{I}  \equiv 0 \pmod{I}
\end{align*}
because, by Case 2, the polynomials of the form (\ref{prod_a_43}) of degree $m$ belong to $I$. It follows that 
\begin{align*}
[a_{1} , a_{2} , a_{3} ]\bigl( [a_4,a_5][a_6 , a_7]+[a_{4} , a_6][a_5 , a_{7} ]\bigr) & \equiv [a_{1} , a_{2} , a_{3} ]\bigl([a_5,a_4][a_7 , a_6]+ [a_5 , a_{7} ] [a_{4} , a_6]\bigr) \pmod{I}
\\
& \equiv  [a_{1} , a_{2} , a_{3} ]\bigl([a_{6} , a_7][a_4,a_5] + [a_6 , a_4][a_7 , a_5 ]\bigr) \pmod{I} 
\\
& \equiv [a_{1} , a_{2} , a_{3} ]\bigl( [a_7 , a_6][a_5,a_4] +[a_7 , a_5 ][a_6 , a_4] \bigr) \pmod{I} ,
\end{align*}
that is, 
\begin{align}
\label{case2_5678}
&f(a_1,a_2,a_3,a_4,a_5,a_6,a_7) \equiv f(a_1,a_2,a_3, a_5,a_4,a_7,a_6) \pmod{I}
\\
\equiv & f(a_1,a_2,a_3,a_6,a_7,a_4,a_5) \pmod{I} \equiv f(a_1,a_2,a_3,a_7,a_6,a_5,a_4) \pmod{I}.
\nonumber
\end{align}
It follows from (\ref{case2_1234}) and (\ref{case2_5678}) that one may assume without loss of generality that $i \in \{ 1, 3,4 \} $, as claimed.

Suppose first that $a_{1}=a_{1}'a_{1}''$. Let $g = \bigl([a_4,a_5][a_{6} , a_7]+[a_{4} , a_6][a_5 , a_{7} ]\bigr).$ By (\ref{comm_exp}), we have
\begin{align}
\label{case2_1}
f = & [a_{1}' a_1'' , a_{2} , a_{3} ] g
\\
= & \bigl( a_1' [a_1'',a_2,a_3] 
+ [a_1', a_3] [a_1'', a_2] + [a_1' , a_2] [a_1'', a_3]  
+ a_1'' [a_1' , a_2, a_3] 
+ [a_1' , a_2, a_3, a_1''] \bigr) g .
\nonumber
\end{align}
By the induction hypothesis, 
\[
[b,a_2,a_3] g = [b,a_2,a_3] \bigl([a_4,a_5][a_{6} , a_7]+[a_{4} , a_6][a_5 , a_{7} ]\bigr) \in I
\] 
if $b \in \{ a_1', a_1'' \}$ so 
\begin{equation}
\label{case2_11}
\bigl( a_1' [a_1'',a_2,a_3] + a_1'' [a_1' , a_2, a_3] \bigr) g \in I.
\end{equation}
Further,  
\begin{align}
\label{case2_12}
& \bigl( [a_1', a_3] [a_1'', a_2] + [a_1' , a_2] [a_1'', a_3] \bigr) g = - \bigl( [a_1', a_3] [a_2, a_1''] + [a_1' , a_2] [a_3, a_1''] \bigr) g
\\
= & - \bigl( [a_1', a_3] [a_2, a_1''] + [a_1' , a_2] [a_3, a_1'']  \bigr)  \bigl([a_4,a_5][a_{6} , a_7]+[a_{4} , a_6][a_5 , a_{7} ]\bigr) \in I
\nonumber
\end{align}
because this polynomial is (up to sign) of the form (\ref{prod_a_2222}) and of degree $m$ so, by Case 1, it belongs to $I$. Finally, by Lemma \ref{lemma_422},
\begin{equation}
\label{case2_13}
[a_1', a_2, a_3, a_1''] g = [a_1' , a_2, a_3, a_1''] \bigl([a_4,a_5][a_{6} , a_7]+[a_{4} , a_6][a_5 , a_{7} ]\bigr)  \in I.
\end{equation}
It follows from (\ref{case2_1}), (\ref{case2_11}), (\ref{case2_12}) and (\ref{case2_13}) that $f \in I$ if $a_1 = a_1'a_1''$. 

Suppose that $a_3=a_3'a_3''$. Then
\begin{align*}
f = & [a_{1} , a_{2} , a_{3}' a_3'' ] \bigl([a_4,a_5][a_{6} , a_7]+[a_{4} , a_6][a_5 , a_{7} ]\bigr) 
\\ 
= & \bigl( a_3' [a_1,a_2,a_3''] + a_3'' [a_1,a_2,a_3'] + [a_1,a_2,a_3',a_3''] \bigr) \bigl([a_4,a_5][a_{6} , a_7]+[a_{4} , a_6][a_5 , a_{7} ]\bigr)
\end{align*}
where $\bigl( a_3' [a_1,a_2,a_3''] + a_3'' [a_1,a_2,a_3'] \bigr) \bigl([a_4,a_5][a_{6} , a_7]+[a_{4} , a_6][a_5 , a_{7} ]\bigr) \in I$ by the induction hypothesis and $[a_1,a_2,a_3',a_3''] \bigl([a_4,a_5][a_{6} , a_7]+[a_{4} , a_6][a_5 , a_{7} ]\bigr) \in I$ by Lemma \ref{lemma_422}. It follows that $f \in I$ if $a_3 = a_3' a_3''$.

Finally, suppose that $a_4 = a_4' a_4''$. We have
\begin{align*}
f = & [a_{1} , a_{2} , a_3 ] \bigl([a_4' a_4'',a_5][a_{6} , a_7]+[a_4' a_4'' , a_6][a_5 , a_{7} ]\bigr) 
\\
= & [a_{1} , a_{2} , a_3 ] \bigl( a_4' [a_4'',a_5][a_{6} , a_7] + [a_4' ,a_5] a_4'' [a_{6} , a_7] + a_4' [a_4'' , a_6][a_5 , a_{7} ] + [a_4'  , a_6] a_4''  [a_5 , a_{7} ]  \bigr) 
\\
= & a_4' [a_{1} , a_{2} , a_3 ] \bigl(  [a_4'',a_5][a_{6} , a_7] + [a_4'' , a_6][a_5 , a_{7} ]  \bigr) + [a_{1} , a_{2} , a_3, a_4'  ] \bigl(  [a_4'',a_5][a_{6} , a_7] + [a_4'' , a_6][a_5 , a_{7} ]  \bigr) 
\\
+ & a_4'' [a_{1} , a_{2} , a_3 ] \bigl(  [a_4',a_5][a_{6} , a_7] + [a_4' , a_6][a_5 , a_{7} ]  \bigr) + [a_{1} , a_{2} , a_3, a_4''  ] \bigl(  [a_4',a_5][a_{6} , a_7] + [a_4' , a_6][a_5 , a_{7} ]  \bigr) 
\\
+ & [a_{1} , a_{2} , a_3 ] \bigl(  [a_4' ,a_5, a_4''] [a_{6} , a_7]  + [a_4'  , a_6, a_4'']  [a_5 , a_{7} ]  \bigr) . 
\end{align*}
Here
\begin{align*}
& a_4' [a_{1} , a_{2} , a_3 ] \bigl(  [a_4'',a_5][a_{6} , a_7] + [a_4'' , a_6][a_5 , a_{7} ]  \bigr) +  a_4'' [a_{1} , a_{2} , a_3 ] \bigl(  [a_4',a_5][a_{6} , a_7] + [a_4' , a_6][a_5 , a_{7} ]  \bigr) \\
& +  [a_{1} , a_{2} , a_3 ] \bigl(  [a_4' ,a_5, a_4''] [a_{6} , a_7]  + [a_4'  , a_6, a_4'']  [a_5 , a_{7} ]  \bigr)  \in I
\end{align*}
by the induction hypothesis and 
\[
 [a_{1} , a_{2} , a_3, a_4'  ] \bigl(  [a_4'',a_5][a_{6} , a_7] + [a_4'' , a_6][a_5 , a_{7} ]  \bigr) + [a_{1} , a_{2} , a_3, a_4''  ] \bigl(  [a_4',a_5][a_{6} , a_7] + [a_4' , a_6][a_5 , a_{7} ]  \bigr)  \in I
\]
by Lemma \ref{lemma_422}. It follows that $f \in I$ if $a_4 = a_4' a_4''$.

Thus, each polynomial of the form (\ref{prod_a_322}) and of degree $m$ belongs to $I$.


\textbf{Case 4.} Suppose now that $f$ is of the form (\ref{prod_a_2212}),
\begin{align*}
f = & f(a_1,a_2, \dots , a_7)
\\
= & \Big[ \big( [a_{1},a_{2}][a_{3},a_{4}]+[a_{1},a_{3}][a_{2},a_{4}] \big) ,a_{5} \Big] [a_{6},a_{7}]
+ \Big[ \big( [a_{1},a_{2}][a_{3},a_{4}]+[a_{1},a_{3}][a_{2},a_{4}] \big) ,a_{6} \Big] [a_{5},a_{7}] .
\end{align*}
If $\deg f = 7$ then $f$ is of the form (\ref{prod_x_2212}) so $f \in I$. If $\deg f>7$ then for some $i$, $1\le i \le 7$, we have $a_i=a_i'a_i''$, where $\deg a_i'$, $\deg a_i'' < \deg a_i$. We claim that to check that $f \in I$ one may assume without loss of generality that $i \in \{ 4, 6,7 \} $.

Indeed, it is clear that 
\begin{equation}
\label{case4_56}
f(a_1, \dots, a_4, a_5, a_6, a_7) = f(a_1, \dots , a_4, a_6, a_5, a_7).
\end{equation}
On the other hand, if $\{ k, \ell , k', \ell' \} = \{ 1,2,3,4 \}$ then, by the induction hypothesis, 
\[
\bigl[ [a_{k},a_{\ell}] , [a_{k'},a_{\ell'}] ,a_{5} \bigr ] =  [a_{k},a_{\ell} , a_{k'},a_{\ell'} ,a_{5}  ] - [a_{k},a_{\ell} , a_{\ell'},a_{k'} ,a_{5}  ] \in I .
\]
It follows that 
\begin{align}
\label{case4_234}
& f(a_1,a_2,a_3,a_4,a_5, \dots ) \equiv f(a_2,a_1,a_4,a_3,a_5, \dots ) \pmod{I} 
\\
\equiv & f(a_3,a_4,a_1,a_2,a_5, \dots ) \pmod{I} \equiv f(a_4,a_3,a_2,a_1, a_5, \dots ) \pmod{I}.
\nonumber
\end{align}
It follows from (\ref{case4_56}) and (\ref{case4_234}) that one may assume without loss of generality that $i \in \{ 4, 6,7 \} $, as claimed.

Suppose that $a_4 = a_4' a_4''$. Then
\begin{align}
\label{case44_4}
& f =  f(a_1,a_2, a_3, a_4' a_4'', a_5, a_6 , a_7)
\\
& =  \Big[ \big( [a_{1},a_{2}][a_{3},a_{4}' a_4'']+[a_{1},a_{3}][a_{2},a_{4}' a_4''] \big) ,a_{5} \Big] [a_{6},a_{7}]
+ \Big[ \big( [a_{1},a_{2}][a_{3},a_{4}' a_4'']+[a_{1},a_{3}][a_{2},a_{4}' a_4''] \big) ,a_{6} \Big] 
[a_{5},a_{7}] 
\nonumber
\\
& = \Big[ \big( [a_{1},a_{2}][a_{3},a_{4}' ]+[a_{1},a_{3}][a_{2},a_{4}' ] \big) a_4'' ,a_{5} \Big] [a_{6},a_{7}]
+ \Big[ \big( [a_{1},a_{2}][a_{3},a_{4}'' ]+[a_{1},a_{3}][a_{2},a_{4}'' ] \big) a_4' ,a_{5} \Big] [a_{6},a_{7}]
\nonumber
\\
& -  \Big[ \big( [a_{1},a_{2}][a_{3}, a_4'', a_{4}']+[a_{1},a_{3}][a_{2}, a_4'', a_{4}'] \big) ,a_{5} \Big] [a_{6},a_{7}] + \Big[ \big( [a_{1},a_{2}][a_{3},a_{4}' ]+[a_{1},a_{3}][a_{2},a_{4}' ] \big) a_4'' ,a_{6} \Big] [a_{5},a_{7}]
\nonumber
\\
& + \Big[ \big( [a_{1},a_{2}][a_{3},a_{4}'' ]+[a_{1},a_{3}][a_{2},a_{4}'' ] \big) a_4' ,a_{6} \Big] [a_{5},a_{7}] -  \Big[ \big( [a_{1},a_{2}][a_{3}, a_4'', a_{4}']+[a_{1},a_{3}][a_{2}, a_4'', a_{4}'] \big) ,a_{6} \Big] [a_{5},a_{7}]
\nonumber
\end{align}
We have
\begin{align}
\label{case44_41}
& \Big[ \big( [a_{1},a_{2}][a_{3}, a_4'', a_{4}']+[a_{1},a_{3}][a_{2}, a_4'', a_{4}'] \big) ,a_{5} \Big] [a_{6},a_{7}] 
\\
 + & \Big[ \big( [a_{1},a_{2}][a_{3}, a_4'', a_{4}']+[a_{1},a_{3}][a_{2}, a_4'', a_{4}'] \big) ,a_{6} \Big] [a_{5},a_{7}]
\nonumber
\\
 =  & \big( [a_{1},a_{2},a_5][a_{3}, a_4'', a_{4}']+[a_{1},a_{3}, a_5][a_{2}, a_4'', a_{4}'] \big) [a_{6},a_{7}] +  \big( [a_{1},a_{2},a_6][a_{3}, a_4'', a_{4}'] 
\nonumber
\\
+ & [a_{1},a_{3}, a_6][a_{2}, a_4'', a_{4}'] \big) [a_{5},a_{7}] +  [a_{1},a_{2}] \bigl( [a_{3}, a_4'', a_{4}' ,a_5] [a_{6},a_{7}] + [a_{3}, a_4'', a_{4}' a_6][a_{5},a_{7}] \bigr)
\nonumber
\\
+ & [a_{1},a_{3}] \bigl( [a_{2}, a_4'', a_{4}', a_5]  [a_{6},a_{7}] + [a_{2}, a_4'', a_{4}', a_6]   [a_{5},a_{7}] \bigr) \in I
\nonumber
\end{align}
because, by the induction hypothesis,  
\[
[a_{k}, a_4'', a_{4}', a_{\ell}]  [a_{\ell'},a_{7}] + [a_{k}, a_4'', a_{4}', a_{\ell'}]   [a_{\ell},a_{7}] \in I
\]
and $[a_{1},a_{k},a_{\ell}][a_{k'}, a_4'', a_{4}'] \in I$ if $\{ k, k' \} = \{ 2,3 \}$ and $\{ \ell, \ell' \} = \{5,6 \}$.

Further,
\begin{align}
\label{case44_42}
& \Big[ \big( [a_{1},a_{2}][a_{3},a_{4}' ]+[a_{1},a_{3}][a_{2},a_{4}' ] \big) a_4'' ,a_{5} \Big] [a_{6},a_{7}] + \Big[ \big( [a_{1},a_{2}][a_{3},a_{4}' ]+[a_{1},a_{3}][a_{2},a_{4}' ] \big) a_4'' ,a_{6} \Big] [a_{5},a_{7}]
\\
& =a_4'' \bigg( \Big[ \big( [a_{1},a_{2}][a_{3},a_{4}' ]+[a_{1},a_{3}][a_{2},a_{4}' ] \big)  ,a_{5} \Big] [a_{6},a_{7}]  + \Big[ \big( [a_{1},a_{2}][a_{3},a_{4}' ]+[a_{1},a_{3}][a_{2},a_{4}' ] \big) ,a_{6} \Big] [a_{5},a_{7}] \bigg)
\nonumber
\\
& + \Big[ \big( [a_{1},a_{2}][a_{3},a_{4}' ]+[a_{1},a_{3}][a_{2},a_{4}' ] \big)  ,a_{5}, a_4'' \Big] [a_{6},a_{7}]  +  \Big[ \big( [a_{1},a_{2}][a_{3},a_{4}' ]+[a_{1},a_{3}][a_{2},a_{4}' ] \big) ,a_{6}, a_4'' \Big] [a_{5},a_{7}]
\nonumber
\\
& + \big( [a_{1},a_{2}][a_{3},a_{4}' ]+[a_{1},a_{3}][a_{2},a_{4}' ] \big) \bigl( [a_4'',a_5][a_6,a_7] + [a_4'',a_6] [a_5,a_7] \bigr) \in I
\nonumber
\end{align}
because the polynomials of types (\ref{prod_a_2212}) and (\ref{comm_a_2211}) of degree less than $m$ belong to $I$ by the induction hypothesis and the polynomial of type (\ref{prod_a_2222}) of degree $m$ belongs to $I$ by Case 1. Similarly, 
\begin{equation}
\label{case44_43}
\Big[ \big( [a_{1},a_{2}][a_{3},a_{4}'' ]+[a_{1},a_{3}][a_{2},a_{4}'' ] \big) a_4' ,a_{5} \Big] [a_{6},a_{7}] + \Big[ \big( [a_{1},a_{2}][a_{3},a_{4}'' ]+[a_{1},a_{3}][a_{2},a_{4}'' ] \big) a_4' ,a_{6} \Big] [a_{5},a_{7}] \in I.
\end{equation}
It follows from (\ref{case44_4}), (\ref{case44_41}), (\ref{case44_42}) and (\ref{case44_43}) that $f \in I$ if $a_4 = a_4' a_4''$.

Now suppose that $a_6 = a_6' a_6''$. Let $g = [a_{1},a_{2}][a_{3},a_{4}]+[a_{1},a_{3}][a_{2},a_{4}]$. We have
\begin{align}
\label{case44_6}
f &=  [g ,a_{5} ] [a_{6}' a_6'',a_{7}] + [ g ,a_{6}' a_6'' ] [a_{5},a_{7}] 
= [ g ,a_{5} ] a_{6}' [a_6'',a_{7}]  + [ g ,a_{5} ] [a_{6}' ,a_{7}] a_6'' +  a_{6}' [ g , a_6'' ] [a_{5},a_{7}] 
\\
& +  [ g ,a_{6}'  ] a_6'' [a_{5},a_{7}] = a_{6}' \bigl(  [ g ,a_{5} ] [a_6'',a_{7}]  +   [ g , a_6'' ] [a_{5},a_{7}] \bigr) + [ g ,a_{5} , a_{6}' ] [a_6'',a_{7}]  + \bigl( [ g ,a_{5} ] [a_{6}' ,a_{7}] 
\nonumber
\\
& + [ g ,a_{6}'  ] [a_{5},a_{7}]  \bigr) a_6'' - [ g ,a_{6}'  ]  [a_{5},a_{7}, a_6''] .
\nonumber
\end{align}
Here
\begin{align}
\label{case44_61}
& a_{6}' \bigl(  [ g ,a_{5} ] [a_6'',a_{7}]  +   [ g , a_6'' ] [a_{5},a_{7}] \bigr) + [ g ,a_{5} , a_{6}' ] [a_6'',a_{7}]  + \bigl( [ g ,a_{5} ] [a_{6}' ,a_{7}]  + [ g ,a_{6}'  ] [a_{5},a_{7}]  \bigr) a_6''
\\
= & a_{6}'\bigg(  \Big[ \big( [a_{1},a_{2}][a_{3},a_{4}]+[a_{1},a_{3}][a_{2},a_{4}] \big) ,a_{5} \Big] [a_6'',a_{7}]  +   \Big[ \big( [a_{1},a_{2}][a_{3},a_{4}]+[a_{1},a_{3}][a_{2},a_{4}] \big) , a_6'' \Big] 
[a_{5},a_{7}] \bigg)
\nonumber
\\
+ & \Big[ \big( [a_{1},a_{2}][a_{3},a_{4}]+[a_{1},a_{3}][a_{2},a_{4}] \big) ,a_{5} , a_{6}' \Big] [a_6'',a_{7}] 
+ \bigg( \Big[ \big( [a_{1},a_{2}][a_{3},a_{4}]+[a_{1},a_{3}][a_{2},a_{4}] \big) ,a_{5} \Big] [a_{6}' ,a_{7}] 
\nonumber
\\
+ & \Big[ \big( [a_{1},a_{2}][a_{3},a_{4}]+[a_{1},a_{3}][a_{2},a_{4}] \big) ,a_{6}'  \Big] [a_{5},a_{7}]  \bigg) a_6'' \in I
\nonumber
\end{align}
because polynomials of types (\ref{prod_a_2212}) and (\ref{comm_a_2211}) of degree less than $m$ belong to $I$ by the induction hypothesis. Further, 
\begin{align}
\label{case44_62}
 & [ g ,a_{6}'  ]  [a_{5},a_{7}, a_6''] = \Big[ \big( [a_{1},a_{2}][a_{3},a_{4}]+[a_{1},a_{3}][a_{2},a_{4}] \big) ,a_{6}'  \Big]  [a_{5},a_{7}, a_6'']
 \\
 = &  \bigl( [a_{1},a_{2}][a_{3},a_{4}, a_{6}' ] + [a_{1},a_{2}, a_{6}' ][a_{3},a_{4} ] +  [a_{1},a_{3}][a_{2},a_{4} ,a_{6}' ] + [a_{1},a_{3}, ,a_{6}' ][a_{2},a_{4} ] \bigr) [a_{5},a_{7}, a_6'']
\nonumber
\\
\equiv & \Big( [a_{1},a_{2}][a_{3},a_{4}, a_{6}' ] +    [a_{3},a_{4} ] [a_{1},a_{2}, a_{6}' ] +  [a_{1},a_{3}][a_{2},a_{4} ,a_{6}' ] 
\nonumber
\\
+ & [a_{2},a_{4} ] [a_{1},a_{3}, ,a_{6}' ] \Big) [a_{5},a_{7}, a_6''] \pmod{I} \equiv 0 \pmod{I}
\nonumber 
\end{align}
because the polynomials of the form (\ref{prod_a_33}) of degree less than $m$ belong to $I$ by the induction hypothesis. It follows from (\ref{case44_6}), (\ref{case44_61}) and  (\ref{case44_62}) that $f \in I$ if $a_6 = a_6' a_6''$.

Finally, suppose that $a_7 = a_7' a_7''$. Then
\begin{align*}
\label{case44_7}
f = & [ g ,a_{5} ] [a_6, a_{7}' a_7''] + [ g ,a_{6} ] [a_{5},a_{7}' a_7'']  =  [ g ,a_{5} ] a_7' [a_6, a_7''] +  [ g ,a_{5} ] [a_6, a_{7}' ] a_7''  + [ g ,a_{6} ] a_{7}' [a_{5}, a_7'']  
\\
+ & [ g ,a_{6} ] [a_{5},a_{7}' ] a_7'' = \bigl( [ g ,a_{5} ] [a_6, a_7''] + [ g ,a_{6} ]  [a_{5}, a_7''] \bigr)  a_{7}' + \bigl( [ g ,a_{5} ] [a_6, a_{7}' ]  + [ g ,a_{6} ] [a_{5},a_{7}' ] \bigr) a_7'' 
\nonumber
\\
- & [ g ,a_{5} ] [a_6, a_7'', a_7'] - [ g ,a_{6} ]  [a_{5}, a_7'', a_7'] .
\nonumber
\end{align*}
One can prove repeating the argument used for $i=6$ (when $a_6=a_6' a_6''$) that $f \in I$ if $a_7=a_7' a_7''$.

Thus, each polynomial $f$ of the form (\ref{prod_a_2212})  and of degree $m$ belongs to $I$.


\medskip
\textbf{Case 5.} Suppose that $f$ is of the form (\ref{comm_a_2211}),
\[
f = \Big[ \big( [a_1, a_2] [a_3, a_4] + [a_1, a_3] [a_2, a_4] \big) , a_5, a_6 \Big] .
\]
If $\deg f = 6$ then each  monomial $a_i$ is of degree $1$ so $f$ is of the form (\ref{comm_x_2211}) and, therefore, $f \in I$. If $\deg f>6$ then, for some $i$, $1\le i \le 6$, we have $a_i=a_i' a_i''$ where $\deg a_i'$, $\deg a_i''< \deg a_i$.

We claim that to check that $f \in I$ one may assume without loss of generality that $ i \in \{ 4,5,6 \}$. Indeed, it follows from the induction hypothesis that if $\{ k,\ell ,k', \ell' \} = \{ 1,2,3,4 \}$ then
\[
\big[ [a_{k}, a_{\ell}], [a_{k'}, a_{\ell'}], a_5 \big] =
[a_{k}, a_{\ell}, a_{k'}, a_{\ell'}, a_5 ] - [a_{k}, a_{\ell}, a_{k'}, a_{\ell'}, a_5 ]
\in I .
\]
Hence, $\big[ [a_{k}, a_{\ell}], [a_{k'}, a_{\ell'}], a_5 , a_6 \big] \in I$. 
It follows that
\begin{align*}
\Big[ \bigl( [a_1,a_2] [a_3,a_4] + [a_1,a_3][a_2,a_4] \bigr) , a_5, a_6 \Big] & \equiv \Big[ \bigl( [a_2,a_1] [a_4,a_3] + [a_2,a_4][a_1,a_3] \bigr) , a_5, a_6 \Big]  \pmod{I}
\\
& \equiv  \Big[ \bigl( [a_3,a_4] [a_1,a_2] + [a_3,a_1][a_4,a_2] \bigr) ,a_5, a_6 \Big] \pmod{I} 
\\
& \equiv    \Big[ \bigl( [a_4,a_3] [a_2,a_1] + [a_4,a_2][a_3,a_1] \bigr) , a_5, a_6 \Big] \pmod{I} ,
\end{align*}
that is, 
\begin{align*}
\label{case1_234}
& f(a_1,a_2,a_3,a_4,a_5, a_6 ) \equiv f(a_2,a_1,a_4,a_3,a_5, a_6 ) \pmod{I} 
\\
\equiv & f(a_3,a_4,a_1,a_2,a_5, a_6 ) \pmod{I} \equiv f(a_4,a_3,a_2,a_1, a_5, a_6 ) \pmod{I}.
\nonumber
\end{align*}
The claim follows.

Suppose first that $a_{6}=a_6' a_6''$. Let $g=\big( [a_{1},a_{2}][a_{3},a_{4}]+[a_{1},a_{3}][a_{2},a_{4}] \big)$.  We have
\begin{align*}
f  = & [ g ,a_{5},a_6'a_6'' ] = a_6' [ g ,a_{5},a_6'' ] + [ g ,a_{5},a_6' ] a_6''.
\end{align*}
By the induction hypothesis, $[g , a_{5}, b ] = \Big[ \big( [a_{1},a_{2}][a_{3},a_{4}]+[a_{1},a_{3}][a_{2},a_{4}] \big) , a_{5}, b \Big] \in I$ for $b \in \{a_6',a_6''\}$ so in this case $f \in I$.

Suppose that $a_5=a_5'a_5''$. Then
\begin{align*}
f & = [ g , a_5'a_5'', a_{6} ] = \bigl[ a_5'[ g , a_5''] +  [ g , a_5'] a_5'', a_{6} \bigr]
\\
& = a_5' [ g , a_5'', a_{6} ] + [a_5',a_6] [ g , a_5'' ] + [ g , a_5' ] [a_5'',a_{6}] + [ g , a_5', a_{6} ] a_5'.
\end{align*}
By the induction hypothesis, $[ g , b, a_{6} ] = \Big[ \big( [a_{1},a_{2}] [a_{3},a_{4}] + [a_{1},a_{3}] [a_{2},a_{4}] \big) , b, a_{6} \Big] \in I$ for $b \in \{a_5',a_5''\}$. On the other hand,
\begin{align*}
& [a_5',a_6] [ g , a_5'' ] + [ g , a_5' ] [a_5'',a_{6}] =  [ g , a_5'' ] [a_5',a_6] + [ g , a_5' ] [a_5'', a_{6}] + \bigl[ g , a_5'',[a_5',a_6] \bigr] .
\end{align*}
Since each polynomial of degree $m$ of the form (\ref{prod_a_2212}) belongs to $I$, we have
\begin{align*}
& [ g , a_5' ] [a_5'', a_{6}] + [ g , a_5'' ] [a_5',a_6]
\\
= & \Big[ \big( [a_{1},a_{2}][a_{3},a_{4}]+[a_{1},a_{3}][a_{2},a_{4}] \big) , a_5' \Big] [a_5'',a_{6}]
+ \Big[ \big( [a_{1},a_{2}][a_{3},a_{4}]+[a_{1},a_{3}][a_{2},a_{4}] \big) , a_5'' \Big][a_5',a_6] \in I
\end{align*}
and, by the induction hypothesis,
\begin{align*}
& \bigl[ g , a_5'',[a_5',a_6] \bigr] = [ g , a_5'', a_5', a_6] -  [ g , a_5'', a_6 , a_5']  
\\
= & \Big[ \big( [a_{1},a_{2}][a_{3},a_{4}] + [a_{1},a_{3}][a_{2},a_{4}] \big) , a_5'', a_5', a_6 \Big]
- \Big[ \big( [a_{1},a_{2}][a_{3},a_{4}] + [a_{1},a_{3}][a_{2},a_{4}] \big) ,a_5'', a_6,  a_5' \Big] \in I.
\end{align*}
Hence, in this case $f \in I$.

Now suppose that $a_4=a_4'a_4''$. We have
\begin{align*}
f = & \Big[ \big( [a_{1},a_{2}][a_{3},a_4'a_4''] + [a_{1},a_{3}][a_{2},a_4'a_4''] \big) , a_{5}, a_{6} \Big]
\\
= & \Big[ \big( [a_{1},a_{2}] a_4'[a_{3},a_4''] + [a_{1},a_{3}] a_4' [a_{2},a_4''] \big), a_{5}, a_{6} \Big]
+ \Big[ \big( [a_{1},a_{2}] [a_{3},a_4'] a_4'' + [a_{1},a_{3}][a_{2},a_4']a_4'' \big) , a_{5}, a_{6} \Big]
\\
= & \Big[ a_4' \big( [a_{1},a_{2}] [a_{3},a_4''] + [a_{1},a_{3}] [a_{2},a_4'']\big) , a_{5}, a_{6} \Big]
+ \Big[ \big( [a_{1},a_{2}] [a_{3},a_4']+[a_{1},a_{3}] [a_{2},a_4'] \big) a_4'', a_{5}, a_{6}\Big]
\\
+ & \Big[ \big( [a_{1},a_{2},a_4'] [a_{3},a_4''] + [a_{1},a_{3},a_4'][a_{2},a_4''] \big) , a_{5}, a_{6} \Big] .
\end{align*}
By (\ref{comm_exp}), we have
\begin{align*}
& \Big[ a_4' \big( [a_{1},a_{2}] [a_{3},a_4''] + [a_{1},a_{3}] [a_{2},a_4'']\big) , a_{5}, a_{6} \Big]
\\
= & a_4' \Big[ \big( [a_{1},a_{2}] [a_{3},a_4''] + [a_{1},a_{3}] [a_{2},a_4'']\big) , a_{5}, a_{6} \Big]
+ [ a_4', a_5] \Big[ \big( [a_{1},a_{2}] [a_{3},a_4''] + [a_{1},a_{3}] [a_{2},a_4'']\big) , a_{6} \Big]
\\
+ & [a_4', a_6] \Big[ \big( [a_{1},a_{2}] [a_{3},a_4''] + [a_{1},a_{3}] [a_{2},a_4'']\big) , a_{5} \Big]
+ [a_4', a_5, a_6] \big( [a_{1},a_{2}] [a_{3},a_4''] + [a_{1},a_{3}] [a_{2},a_4'']\big)
\\
= & a_4' \Big[ \big( [a_{1},a_{2}] [a_{3},a_4''] + [a_{1},a_{3}] [a_{2},a_4'']\big) , a_{5}, a_{6} \Big]
- \bigg( \Big[ \big( [a_{1},a_{2}] [a_{3},a_4''] + [a_{1},a_{3}] [a_{2},a_4'']\big) , a_{6} \Big] [ a_5, a_4']
\\
+ & \Big[ \big( [a_{1},a_{2}] [a_{3},a_4''] + [a_{1},a_{3}] [a_{2},a_4'']\big) , a_{5} \Big] [a_6, a_4'] \bigg)
+ [a_4', a_5, a_6] \big( [a_{1},a_{2}] [a_{3},a_4''] + [a_{1},a_{3}] [a_{2},a_4'']\big)
\\
- & \bigg[ \Big[ \big( [a_{1},a_{2}] [a_{3},a_4''] + [a_{1},a_{3}] [a_{2},a_4'']\big) , a_{6} \Big] ,[ a_5, a_4'] \bigg]
- \bigg[ \Big[ \big( [a_{1},a_{2}] [a_{3},a_4''] + [a_{1},a_{3}] [a_{2},a_4'']\big) , a_{5} \Big] , [a_6, a_4'] \bigg] .
\end{align*}
By the induction hypothesis, $\Big[ \big( [a_{1},a_{2}][a_{3}, a_4'']+[a_{1},a_{3}][a_{2}, a_4''] \big) ,a_{5},a_{6} \Big] \in I$. Further,
\[
[a_4',a_{5},a_{6}] \big( [a_{1},a_{2}][a_{3},a_4'']+[a_{1},a_{3}][a_{2},a_4''] \big)
\]
and
\begin{align*}
& \Big[ \big( [a_{1},a_{2}] [a_{3},a_4''] + [a_{1},a_{3}] [a_{2},a_4'']\big) , a_{6} \Big] [ a_5, a_4']
+ \Big[ \big( [a_{1},a_{2}] [a_{3},a_4''] + [a_{1},a_{3}] [a_{2},a_4'']\big) , a_{5} \Big] [a_6, a_4']
\end{align*}
are polynomials of degree $m$ of the forms (\ref{prod_a_322}) and (\ref{prod_a_2212}), respectively, so they belong to $I$. Finally,
\begin{align*}
& \bigg[ \Big[ \big( [a_{1},a_{2}] [a_{3},a_4''] + [a_{1},a_{3}] [a_{2},a_4'']\big) , a_{6} \Big] ,[ a_5, a_4'] \bigg]
\\
= & \Big[ \big( [a_{1},a_{2}] [a_{3},a_4''] + [a_{1},a_{3}] [a_{2},a_4'']\big) , a_{6} , a_5, a_4' \Big]
- \Big[ \big( [a_{1},a_{2}] [a_{3},a_4''] + [a_{1},a_{3}] [a_{2},a_4'']\big) , a_{6} , a_4' , a_5 \Big]
\end{align*}
belongs to $I$ by the induction hypothesis and so does the commutator
\[
\bigg[ \Big[ \big( [a_{1},a_{2}] [a_{3},a_4''] + [a_{1},a_{3}] [a_{2},a_4'']\big) , a_{5} \Big] , [a_6, a_4'] \bigg] .
\]
It follows that
\[
\Big[ a_4' \big( [a_{1},a_{2}] [a_{3},a_4''] + [a_{1},a_{3}] [a_{2},a_4'']\big) , a_{5}, a_{6} \Big] \in I.
\]
One can check in a similar way that
\[
\Big[ \big( [a_{1},a_{2}] [a_{3},a_4']+[a_{1},a_{3}] [a_{2},a_4'] \big) a_4'', a_{5}, a_{6}\Big] \in I.
\]
Finally, by (\ref{comm_exp}), the commutator
\begin{equation}
\label{3211inI}
\Big[ \big( [a_{1},a_{2},a_4'] [a_{3},a_4''] + [a_{1},a_{3},a_4'][a_{2},a_4''] \big) , a_{5}, a_{6} \Big]
\end{equation}
can be written as a sum of products of two commutators in $a_4'$, $a_4''$ and $a_i$ $(i \in \{ 1,2,3,5,6\} )$.  The commutators in the products are either of length $2$ and $5$ or of length $3$ and $4$. By the induction hypothesis, the commutators of length $5$ belong to $I$ and so do the products that contain such commutators. The products of commutators of length $3$ and $4$ are, modulo $I$, of the form (\ref{prod_a_43}) and of degree $m$ so, by Case 2, they also belong to $I$. Hence, the commutator (\ref{3211inI}) belongs to $I$ and so does $f$ if $a_4 = a_4' a_4''$.

Thus, each polynomial $f = \Big[ \big( [a_1, a_2] [a_3, a_4] + [a_1, a_3] [a_2, a_4] \big) , a_5, a_6 \Big] $ of the form (\ref{comm_a_2211}) of degree $m$ belongs to $I$.


\medskip
\textbf{Case 6.} Suppose that $f$ is of the form (\ref{prod_a_33}),
\[
f= f(a_1, a_2, \dots , a_6) = [a_1,a_2,a_3][a_4,a_5,a_6].
\]
If $\deg f=6$ then each monomial $a_i$ is of degree 1 so $f$ is of the form (\ref{prod_x_33}) and, therefore, $f \in I$. If $\deg f >6$ then, for some $i$, $1 \le i \le 6$, we have $a_i=a_i'a_i''$ where $\deg a_i'$, $\deg a_i'' < \deg a_i$. We claim that to check that $f \in I$ one may assume without loss of generality that $i = 5$ or $i = 6$. Indeed, it is clear that $f(\dots , a_4, a_5, a_6 ) =  - f(\dots , a_5, a_4, a_6 )$. Also,
\begin{equation}
\label{33inI}
\big[ [a_1,a_2,a_3], [a_4,a_5,a_6] \big] \in I.
\end{equation}
Indeed, this commutator is a linear combination of the (left-normed) commutators of the form 
\[
[a_{k_1}, a_{k_2}, a_{k_3}, a_{k_4} , a_{k_5}, a_{k_6}]
\]
where $\{k_1 , \dots k_5, k_6 \} = \{ 1, \dots ,5,6\}$. Since, by the induction hypothesis, each commutator $[a_{k_1}, \dots , a_{k_5}]$ belongs to $I$, so does each commutator $[a_{k_1}, \dots , a_{k_5},a_{k_6}]$; therefore, (\ref{33inI}) holds. By (\ref{33inI}), we have
\[
f(a_1,a_2,a_3,a_4,a_5,a_6) \equiv f(a_4,a_5,a_6,a_1,a_2,a_3) \pmod{I}.
\]
The claim follows.

Suppose that $a_6=a_6'a_6''$. We have
\begin{align*}
f = & [a_{1} , a_{2} , a_{3} ][a_{4} , a_{5} , a_6' a_6'' ] = [a_{1} , a_{2} , a_{3} ] \big( a_6'[a_{4} , a_{5} ,a_6'' ]+[a_{4} , a_{5} , a_6']a_6''\big)
\\
= & a_6'[a_{1} , a_{2} , a_{3} ][a_{4} , a_{5} ,a_6'' ] + [a_{1} , a_{2} , a_{3} ][a_{4} , a_{5} , a_6']a_6''
+ [a_{1} , a_{2} , a_{3} ,a_6' ][a_{4} , a_{5} ,a_6'' ].
\end{align*}
By the induction hypothesis, $[a_{1} , a_{2} , a_{3} ][a_{4} , a_{5} ,b ]\in I$ if $b \in \{a_6', a_6''\}$. Also,  $[a_{1} , a_{2} , a_{3} , a_6' ][a_{4} , a_{5} ,a_6''] \in I$ because this is a polynomial of the form (\ref{prod_a_43}) and of degree $m$ so, by Case 2, it belongs to $I$. It follows that  $f \in I$.

Suppose that $a_5=a_5' a_5''$. Then
\begin{align*}
f = & [a_{1} , a_{2} , a_{3} ][a_{4} , a_5'a_5'' , a_{6} ] = [a_{1} , a_{2} , a_{3} ] \Big[ \big( a_5'[a_{4} , a_5''] +[a_{4} , a_5']a_5'' \big) , a_{6} \Bigr]
\\
= & [a_{1} , a_{2} , a_{3} ] \big(a_5'[a_{4} , a_5'' , a_{6} ]+[a_5',a_6][a_{4} , a_5'']+[a_{4} , a_5'][a_5'' , a_{6} ]+[a_{4} , a_5' , a_{6} ]a_5''\big).
\end{align*}
As in the previous case,
\begin{align*}
&[a_{1} , a_{2} , a_{3} ] \big( a_5'[a_{4} , a_5'' , a_{6} ] + [a_{4} , a_5' , a_{6} ]a_5'' \big)
=  a_5' [a_{1} , a_{2} , a_{3} ] [a_{4} , a_5'' , a_{6} ]
\\
+ & [a_{1} , a_{2} , a_{3} ] [a_{4} , a_5' , a_{6} ]a_5'' + [a_{1} , a_{2} , a_{3} , a_5'] [a_{4} , a_5'' , a_{6} ] \in I.
\end{align*}
Also,
\begin{align*}
& [a_{1} , a_{2} , a_{3} ] \big( [a_5',a_6][a_{4} , a_5''] + [a_{4} , a_5'][a_5'' , a_{6} ] \big)
= [a_{1} , a_{2} , a_{3} ] \big( [a_5',a_6][a_{4} , a_5''] + [ a_5', a_{4}][ a_{6}, a_5'' ] \big)
\end{align*}
is a polynomial of the form (\ref{prod_a_322}) of degree $m$ so, by Case 3, it belongs to $I$. It follows that in this case $f \in I$.

Thus, each polynomial $f = [a_1,a_2,a_3][a_4,a_5,a_6]$ of the form (\ref{prod_a_33}) of degree $m$ belongs to $I$.


\medskip
\textbf{Case 7.} Suppose that $f$ is of the form (\ref{prod_a_42}),
\[
f = [a_1,a_2,a_3,a_4] [a_5,a_6] + [a_1,a_2,a_3,a_5][a_4,a_6].
\]
If $\deg f = 6$ then each  monomial $a_i$ is of degree $1$ so $f$ is of the form (\ref{prod_x_42}) and, therefore, $f \in I$. If $\deg f>6$ then, for some $i$, $1\le i \le 6$, we have $a_i=a_i' a_i''$ where $\deg a_i'$, $\deg a_i''< \deg a_i$.

Suppose first that $a_{6}=a_6' a_6''$. We have
\begin{align*}
f = & [a_{1}, a_{2}, a_{3}, a_{4} ][a_{5}, a_6'a_6'' ] + [a_{1} , a_{2} , a_{3} , a_{5} ][a_{4}, a_6'a_6'' ]
\\
= & [a_{1}, a_{2}, a_{3}, a_{4} ] \big( a_6'[a_{5}, a_6'' ] +[a_{5}, a_6' ]a_6'' \big)
+ [a_{1} , a_{2} , a_{3} , a_{5} ] \big( a_6'[a_{4}, a_6'' ] +[a_{4}, a_6' ]a_6'' \big)
\\
\equiv & a_6' \big( [a_{1}, a_{2}, a_{3}, a_{4} ][a_{5}, a_6'' ]+[a_{1} , a_{2} , a_{3} , a_{5} ][a_{4}, a_6'' ] \big)
\\
+ & \big( [a_{1}, a_{2}, a_{3}, a_{4} ][a_{5}, a_6' ]+[a_{1} , a_{2} , a_{3} , a_{5} ][a_{4}, a_6' ] \big) a_6'' \pmod{I} .
\end{align*}
By the induction hypothesis, $[a_{1}, a_{2}, a_{3}, a_{4} ][a_{5}, b ]+[a_{1} , a_{2} , a_{3} , a_{5} ][a_{4}, b ] \in I$ for $b \in \{ a_6', a_6'' \}$. It follows that in this case $f \in I$.

Suppose that $a_{5}=a_5'a_5''$. Then
\begin{align*}
f = & [a_{1}, a_{2}, a_{3}, a_{4} ][a_5'a_5'', a_6 ] + [a_{1} , a_{2} , a_{3} , a_5'a_5'' ][a_{4}, a_6 ]
\\
= & [a_{1}, a_{2}, a_{3}, a_{4} ] \big( a_5'[a_5'', a_6 ] + [a_5', a_6 ]a_5'' \big)
+ \big( a_5'[a_{1} , a_{2} , a_{3} ,a_5'' ]+[a_{1} , a_{2} , a_{3} , a_5']a_5'' \big) [a_{4}, a_6 ]
\\
\equiv & a_5' \big( [a_{1}, a_{2}, a_{3}, a_{4} ][a_5'', a_6 ]+[a_{1} , a_{2} , a_{3} ,a_5'' ][a_{4}, a_6 ] \big)
+ \big( [a_{1}, a_{2}, a_{3}, a_{4} ][a_5', a_6 ] +[a_{1} , a_{2} , a_{3} , a_5'][a_{4}, a_6 ] \big) a_5''
\\
- & [a_{1} , a_{2} , a_{3} , a_5'][a_{4}, a_6, a_5'' ] \pmod{I}.
\end{align*}
 We have $[a_{1}, a_{2}, a_{3}, a_{4} ][b, a_6 ]+[a_{1} , a_{2} , a_{3} ,b ][a_{4}, a_6 ] \in I$ for $b \in \{a_5', a_5''\}$ by the induction hypothesis and $[a_{1} , a_{2} , a_{3} , a_5'][a_{4}, a_6, a_5'' ]\in I$ since, by Case 2, the polynomials of the form (\ref{prod_a_43}) of degree $m$ belong to $I$. Hence, in this case $f \in I$. Similarly, $f \in I$ if $a_{4}=a_4' a_4''$.

Suppose that $a_{3}=a_3' a_3''$. Then
\begin{align*}
f = & [a_{1}, a_{2}, a_3'a_3'', a_4 ][a_5, a_6 ] + [a_{1} , a_{2} , a_3'a_3'' , a_5 ][a_{4}, a_6 ]
\\
= & \Big[ \big( a_3' [a_1, a_2, a_3''] + [a_1, a_2, a_3'] a_3'' \big) , a_4 \Big] [a_5, a_6]
+ \Big[ \big( a_3'[a_1, a_2, a_3''] + [a_1, a_2, a_3'] a_3'' \big) , a_5 \Big] [a_4, a_6]
\\
= & a_3' [a_{1}, a_{2}, a_3'', a_4 ][a_5, a_6 ] + a_3' [a_{1} , a_{2} , a_3'' , a_5 ][a_{4}, a_6 ]
+ [a_3',a_4][a_{1}, a_{2}, a_3''][a_5, a_6 ] 
\\
+ & [a_3',a_5][a_{1} , a_{2} , a_3'' ][a_{4}, a_6 ]
+ [a_{1}, a_{2}, a_3'][a_3'', a_4 ][a_5, a_6 ] + [a_{1} , a_{2} , a_3'][a_3'' , a_5 ][a_{4}, a_6 ]
\\
+ & [a_{1}, a_{2}, a_3', a_4 ]a_3''[a_5, a_6 ] + [a_{1} , a_{2} , a_3' , a_5 ]a_3''[a_{4}, a_6 ]
\\
\equiv & a_3' \big( [a_{1}, a_{2}, a_3'', a_4 ][a_5, a_6 ] + [a_{1} , a_{2} , a_3'' , a_5 ][a_{4}, a_6 ] \big)
+ [a_{1}, a_{2}, a_3''] \big( [a_3',a_4][a_5, a_6 ] + [a_3',a_5][a_{4}, a_6 ] \big)
\\
+ & [a_{1}, a_{2}, a_3'] \big( [a_3'', a_4 ][a_5, a_6 ] + [a_3'' , a_5 ][a_{4}, a_6 ] \big)
\\
+ & a_3'' \big( [a_{1}, a_{2}, a_3', a_4 ][a_5, a_6 ] + [a_{1} , a_{2} , a_3' , a_5 ][a_{4}, a_6 ] \big) \pmod{I} .
\end{align*}
By the induction hypothesis, for  $b \in \{ a_3', a_3''\}$ we have $[a_{1}, a_{2}, b, a_4 ][a_5, a_6 ] + [a_{1} , a_{2} , b , a_5 ][a_{4}, a_6 ] \in I$; also we  have $[a_{1}, a_{2}, b_1]([b_2, a_4][a_5, a_6 ] + [b_2 ,a_5][a_{4}, a_6 ])\in I$ since, by Case 3, the polynomials of the form (\ref{prod_a_322}) of degree $m$ belong to $I$. Hence, in this case $f \in I$.

Now suppose that $a_{2}=a_2'a_2''$. Then, by (\ref{comm_exp}),
\begin{align*}
f = & [a_{1}, a_2'a_2'', a_3, a_4 ][a_5, a_6 ] + [a_{1} , a_2'a_2'' , a_3 , a_5 ][a_{4}, a_6 ]
\\
= & \Big[ \big( a_2' [a_1, a_2''] + [a_1, a_2'] a_2'' \big) , a_3, a_4 \Big] [a_5, a_6]
+ \Big[ \big( a_2' [a_1, a_2''] + [a_1, a_2'] a_2'' \big) , a_3, a_5 \Big] [a_4, a_6]
\\
= & \big( a_2'[a_{1}, a_2'', a_3, a_4 ] + [a_2', a_4] [a_1, a_2'', a_3] + [a_1, a_2', a_3] [a_2'', a_4]
+ [a_1, a_2', a_3, a_4] a_2'' \big) [a_5, a_6 ] 
\\
+ & \Big[ \big( [a_2', a_3][a_1, a_2''] + [a_1, a_2'][a_2'', a_3] \big) , a_4 \Big] [a_5, a_6]
+ \big( a_2' [a_1, a_2'', a_3, a_5] + [a_2', a_5][a_1, a_2'', a_3] 
\\
+ & [a_1, a_2', a_3] [a_2'', a_5]
+ [a_1, a_2', a_3, a_5] a_2'' \big) [a_4, a_6] + \Big[ \big( [a_2', a_3][a_1, a_2''] + [a_1, a_2'][a_2'', a_3] \big) , a_5 \Big] [a_4, a_6] .
\end{align*}
It follows that
\begin{align*}
f = & a_2'[a_{1}, a_2'', a_3, a_4 ][a_5, a_6 ] + a_2'[a_{1} , a_2'' , a_3 , a_5 ][a_{4}, a_6 ]
+ [a_2',a_4][a_{1}, a_2'', a_3 ][a_5, a_6 ] 
\\
+ & [a_2',a_5][a_{1} ,a_2'' , a_3 ][a_{4}, a_6 ]
+ [a_{1}, a_2',a_3][a_2'',  a_4 ][a_5, a_6 ] + [a_{1} , a_2',a_3][a_2'', a_5 ][a_{4}, a_6 ]
\\
+ &[a_{1}, a_2', a_3, a_4 ]a_2''[a_5, a_6 ] + [a_{1} , a_2', a_3 , a_5 ]a_2''[a_{4}, a_6 ]
\\
+ & \Big[ \big( [a_2', a_3][a_1, a_2''] + [a_1, a_2'][a_2'', a_3] \big) , a_4 \Big] [a_5, a_6]
+ \Big[ \big( [a_2', a_3][a_1, a_2''] + [a_1, a_2'][a_2'', a_3] \big) , a_5 \Big] [a_4, a_6]
\\
\equiv & a_2'\big( [a_{1}, a_2'', a_3, a_4 ][a_5, a_6 ] + [a_{1} , a_2'' , a_3 , a_5 ][a_{4}, a_6 ] \big)
+ a_2''\big( [a_{1}, a_2', a_3, a_4 ][a_5, a_6 ] + [a_{1} , a_2', a_3 , a_5 ][a_{4}, a_6 ] \big)
\\
+ & \Big[ \big( [a_2', a_3][a_1, a_2''] + [a_2', a_1][a_3, a_2''] \big) , a_4 \Big] [a_5, a_6]
+ \Big[ \big( [a_2', a_3][a_1, a_2''] + [a_2', a_1][a_3, a_2''] \big) , a_5 \Big] [a_4, a_6]
\\
+ & [a_{1}, a_2'', a_3 ]\big( [a_2',a_4] [a_5, a_6 ] + [a_2',a_5][a_{4}, a_6 ] \big) + [a_{1}, a_2',a_3] \big( [a_2'',  a_4 ][a_5, a_6 ] + [a_2'', a_5 ][a_{4}, a_6 ] \big) \pmod{I}.
\end{align*}
By the induction hypothesis, $[a_{1}, b, a_3, a_4 ][a_5, a_6 ] + [a_{1} , b , a_3 , a_5 ][a_{4}, a_6 ] \in I$ for $b \in \{a_2',a_2''\}$.
We have also
\[
 [a_{1}, b_1, a_3 ]\big( [b_2,a_4] [a_5, a_6 ] + [b_2,a_5][a_{4}, a_6 ] \big) \in I
\]
if $\{ b_1,b_2 \} = \{ a_2', a_2'' \}$ and
\begin{align*}
& \big[ [a_2', a_3][a_1, a_2''] + [a_2', a_1][a_3, a_2''] , a_4 \big] [a_5, a_6]
+ \big[ [a_2', a_3][a_1, a_2''] + [a_2', a_1][a_3, a_2''] , a_5 \big] [a_4, a_6] \in I
\end{align*}
because, by Cases 3 and 4, the polynomials of degree $m$ of the forms (\ref{prod_a_322}) and (\ref{prod_a_2212}) belong to $I$. Hence, in this case $f \in I$. 

Thus, if $f$ is a polynomial of the form (\ref{prod_a_42}) of degree $m$ then $f$ belongs to $I$.


\medskip
\textbf{Case 8.} Finally, suppose that $f$ is of the form (\ref{comm_a_5}), $f = [a_1,a_2,a_3,a_4,a_5]$. Since $m = \deg f > 5$, we have $a_i = a'_i a''_i$ for some $i$, $1 \le i \le 5$, where $\deg a'_i, \deg  a''_i < \deg  a_i$.

Suppose that $a_{5}=a_5' a_5''$. Then
\[
f = [a_1, a_2, a_3, a_4 ,a_5'a_5'' ] = a_5' [a_1, a_2, a_3, a_4 ,a_5''] + [a_1, a_2, a_3, a_4 ,a_5'] a_5''
\]
so, by the induction hypothesis, $f \in I$.

Now suppose that $a_{4}=a_4' a_4''$. Then
\begin{align*}
f & = [a_1, a_2, a_3, a_4'a_4'' , a_5] = \Big[ \big( a_4' [a_1, a_2, a_3, a_4''] + [a_1, a_2, a_3, a_4'] a_4''\big) , a_5 \Big]
\\
& = a_4' [a_1, a_2, a_3, a_4'' ,a_5] + [a_4', a_5] [a_1, a_2, a_3, a_4''] + [a_1, a_2, a_3, a_4'] [a_4'', a_5] + [a_1, a_2, a_3, a_4', a_5] a_4'' 
\\
& \equiv a_4' [a_1, a_2, a_3, a_4'', a_5] + [a_1, a_2, a_3, a_4''] [a_4', a_5]
+ [a_{1}, a_{2}, a_{3}, a_{4}'][a_4'', a_5] + [a_{1}, a_{2}, a_{3}, a_4' ,a_{5} ] a_4''  \pmod{I} .
\end{align*}
By the induction hypothesis, we have $[a_{1}, a_{2}, a_{3}, b ,a_{5} ] \in I$ for $b \in \{a_4',a_4''\}$. On the other hand,
\[
[a_{1}, a_{2}, a_{3}, a_{4}''] [a_4', a_5] + [a_{1}, a_{2}, a_{3}, a_{4}'] [a_4'', a_5] \in I
\]
because, by Case 7, the polynomials of the form (\ref{prod_a_42}) of degree $m$ belong to $I$. Hence, in this case $f \in I$.

Further, suppose that $a_{3}=a_3'a_3''$. Then, by (\ref{comm_exp}),
\begin{align*}
f & = [a_{1}, a_{2}, a_3'a_3'', a_{4}, a_{5}] = \Big[ \big( a_3' [a_{1}, a_{2}, a_{3}''] + [a_{1}, a_{2}, a_{3}'] a_3'' \big) ,a_4, a_5 \Big]
\\
& = a_3' [a_{1}, a_{2}, a_3', a_{4}, a_5] + [a_3', a_5] [a_{1}, a_{2}, a_3'', a_{4}] + [a_3', a_4] [a_{1}, a_{2}, a_{3}'', a_5] + [a_3', a_4, a_5] [a_{1}, a_{2}, a_{3}''] 
\\
& + [a_{1}, a_{2}, a_{3}'][a_3'' ,a_4, a_{5}] + [a_{1}, a_{2}, a_{3}', a_{5}] [a_3'', a_4] + [a_{1}, a_{2}, a_3', a_{4}][a_3'', a_{5}] + [a_{1}, a_{2}, a_3', a_{4},a_{5}] a_3''
\\
& \equiv a_3' [a_{1}, a_{2}, a_3', a_{4}, a_5] - [a_{1}, a_{2}, a_3'', a_{4}] [a_5, a_3'] - [a_{1}, a_{2}, a_{3}'', a_5] [a_4, a_3']  + [a_3', a_4, a_5] [a_{1}, a_{2}, a_{3}''] 
\\
& + [a_3'' ,a_4, a_{5}]  [a_{1}, a_{2}, a_{3}'] - [a_{1}, a_{2}, a_{3}', a_{5}] [a_4, a_3''] - [a_{1}, a_{2}, a_3', a_{4}][a_5, a_3''] + [a_{1}, a_{2}, a_3', a_{4},a_{5}] a_3'' \pmod{I}.
\end{align*}
By the induction hypothesis, $[a_{1}, a_{2}, b, a_{4}, a_5]\in I$  for $b \in \{ a_3', a_3'' \}$. We also have
\[
[b_1, a_4, a_5][a_{1}, a_{2}, b_2] \in I
\]
and
\[
[a_{1}, a_{2}, b_1, a_{5}] [a_4, b_2] + [a_{1}, a_{2}, b_1, a_{4}][a_5, b_2] \in I
\]
where $\{ b_1, b_2 \} = \{ a_3', a_3'' \}$ because, by Cases 6 and 7, the polynomials of the forms (\ref{prod_a_33}) and (\ref{prod_a_42}) of degree $m$ belong to $I$. Hence, in this case $f \in I$.

Finally suppose that either $a_1 = a_1' a_1''$ or $a_2 = a'_2 a''_2$. It is clear that without loss of generality we may assume $a_2 = a'_2 a''_2$. Then, by (\ref{comm_exp}),
\begin{align*}
f & = [a_{1}, a_2'a_2'', a_{3}, a_{4} ,a_{5}] = \Big[ \big( a_2' [a_{1}, a_2''] +  [a_{1}, a_2'] a_2'' \big) , a_{3}, a_{4}, a_{5} \Big]
\\
& = \Big[ \big( a_2' [a_{1}, a_2'', a_3] +  [a_2', a_3] [a_{1}, a_2''] + [a_{1}, a_2'] [a_2'', a_3 ] + [a_1, a_2', a_3] a_2'' \big) , a_{4}, a_{5} \Big]
\\
& = \Big[ \big( [a_2', a_3][a_1, a_2''] + [a_1, a_2'][a_2'', a_3] \big) , a_4, a_5 \Big] +
a_2' [ a_1, a_2'', a_3, a_4, a_5] + [a_2', a_5][a_1, a_2'', a_3, a_4] 
\\
& + [a_2', a_4] [a_1, a_2'', a_3, a_5] + [a_2', a_4, a_5] [a_1, a_2'', a_3] +  [a_1, a_2', a_3][ a_2'', a_4, a_5] 
\\
& + [a_1, a_2', a_3, a_5] [a_2'', a_4] + [a_1, a_2', a_3, a_4] [a_2'', a_5] + [a_1, a_2', a_3, a_4, a_5] a_2''
\\
& \equiv \Big[ \big( [a_2', a_3][a_1, a_2''] + [a_2', a_1][a_3, a_2''] \big) , a_4, a_5 \Big] +
a_2' [ a_1, a_2'', a_3, a_4, a_5] - [a_1, a_2'', a_3, a_4] [a_5, a_2']
\\
& - [a_1, a_2'', a_3, a_5] [a_4, a_2'] + [a_2', a_4, a_5] [a_1, a_2'', a_3] +  [a_1, a_2', a_3][ a_2'', a_4, a_5] 
\\
& - [a_1, a_2', a_3, a_5] [a_4, a_2''] - [a_1, a_2', a_3, a_4] [a_5, a_2''] + [a_1, a_2', a_3, a_4, a_5] a_2'' \pmod{I} .
\end{align*}
We have $[a_{1}, b, a_{3}, a_{4} ,a_{5} ] \in I$ for $b\in \{ a_2',a_2''\}$ by the induction hypothesis. Also
\[
 \Big[ \big( [a_2', a_3][a_1, a_2''] + [a_2', a_1][a_3, a_2''] \big) , a_4, a_5 \Big]  \in I, 
\]
\[
[a_2', a_{4},a_5][a_1,a_2'', a_{3} ], \ [a_{1}, a_2', a_{3}][a_2'' ,a_4, a_{5} ]  \in I 
\]
and 
\[
[a_{1}, b_1, a_{3}, a_{4}][a_{5}, b_2 ] + [a_{1}, b_1, a_{3}, a_{5}][a_{4}, b_2 ] \in I
\]
where $\{ b_1, b_2 \} = \{ a_2', a_2'' \}$ because, by Cases 5, 6 and 7, all polynomials of the forms (\ref{comm_a_2211}), (\ref{prod_a_33}) and (\ref{prod_a_42}) of degree $m$ belong to $I$. Hence, in this case also $f \in I$.

Thus, if  $f$ is a polynomial of the form (\ref{comm_a_5}) of degree $m$ then $f$ belongs to $I$. This establishes the induction step and proves that $T^{(5)} \subseteq I$.

The proof of Theorem \ref{TbaseT5} is completed.

\section*{Acknowledgment}

This work was partially supported by CNPq grant 554712/2009--1. The second author was partially supported by CNPq grant 307328/2012--0.


\begin{thebibliography}{99}

\bibitem{AE15}
Nabilar Abughazalah, Pavel Etingof, \textit{On properties of the lower central series of associative algebras}, arXiv:1508.00943

\bibitem{AS01}
Bernhard Amberg, Yaroslav Sysak, \textit{Associative rings whose adjoint semigroup is locally nilpotent}, Archiv der Mathematik (Basel) \textbf{76} (2001), 426--435.

\bibitem{AJ10}
Noah Arbesfeld, David Jordan, \textit{New results on the lower central series quotients of a free associative algebra}, Journal of Algebra \textbf{323} (2010), 1813--1825. arXiv:0902.4899

\bibitem{BB11}
Martina Balagovi\'c, Anirudha Balasubramanian,  \textit{On the lower central series quotients of a graded associative algebra}, Journal of Algebra \textbf{328} (2011), 287--300. arXiv:1004.3735

\bibitem{BJ10}
Asilata Bapat, David Jordan, \textit{Lower central series of free algebras in symmetric tensor categories}, Journal of Algebra \textbf{373} (2013), 299--311. arXiv:1001.1375


\bibitem{BEJKL12}
Surya Bhupatiraju, Pavel Etingof, David Jordan, William Kuszmaul, Jason Li, \textit{Lower central series of a free associative algebra over the integers and finite fields}, Journal of Algebra \textbf{372} (2012), 251--274. arXiv:1203.1893

\bibitem{CFZ15}
Katherine Cordwell, Teng Fei, Kathleen Zhou, \textit{On lower central series quotients of finitely generated algebras over $\mathbb Z$}, Journal of Algebra \textbf{423} (2015), 559--572. arXiv:1309.1237

\bibitem{DK15}
Galina Deryabina, Alexei Krasilnikov, \textit{The torsion subgroup of the additive group of a Lie nilpotent associative ring of class $3$,} Journal of Algebra \textbf{428} (2015), 230--255. arXiv:1308.4172

\bibitem{DE08}
Galyna Dobrovolska, Pavel Etingof, \textit{An upper bound for the lower central series quotients of a free associative algebra}, International Mathematics Research Notices \textbf{2008}, no. 12, Art. ID rnn039, 10 pp.  arXiv:0801.1997

\bibitem{DKM08}
Galyna Dobrovolska, John Kim, Xiaoguang Ma, \textit{On the lower central series of an associative algebra (with an appendix by Pavel Etingof)}, Journal of Algebra \textbf{320} (2008), 213--237.  arXiv:0709.1905

\bibitem{EKM09}
Pavel Etingof, John Kim, Xiaoguang Ma, \textit{On universal Lie nilpotent associative algebras}, Journal of Algebra \textbf{321} (2009), 697--703. arXiv:0805.1909

\bibitem{FS07}
Boris Feigin, Boris Shoikhet, \textit{On $[A, A]/[A, A, A]$ and on a $W_n$-action on the consecutive commutators of free asociative algebras}, Mathematical Research Letters  \textbf{14} (2007), 781--795. arXiv:math/0610410

\bibitem{Giambruno-Koshlukov01}
Antonio Giambruno, Plamen Koshlukov, \textit{On the identities of the Grassmann algebra in characteristic $p>0$ }, Israel Journal of Mathematics \textbf{122} (2001), 305--316.

\bibitem{Gordienko07}
A.S. Gordienko, \textit{Codimensions of commutators of length 4}, Russian Mathematical Surveys \textbf{62} (2007), 187--188.

\bibitem{GP15}
A.V. Grishin, S.V. Pchelintsev, \textit{On the centers of relatively free algebras with an identity of Lie nilpotency}, to appear in Sbornik. Mathematics.

\bibitem{GL83}
Narain Gupta, Frank Levin, \textit{On the Lie ideals of a ring}, Journal of Algebra \textbf{81} (1983), 225--231.

\bibitem{GK99}
C.K. Gupta, A.N. Krasil'nikov, \textit{A solution of a problem of Plotkin an Vovsi and an application to varieties of groups}, Journal of the Australian Mathematical Society (Series A) \textbf{67} (1999), 329--355.

\bibitem{Jennings47}
S.A. Jennings, \textit{On rings whose associated Lie rings are nilpotent}, Bulletin of the American Mathematical Society \textbf{53} (1947), 593--597.

\bibitem{JO13}
David Jordan, Hendrik Orem, \textit{An algebro-geometric construction of lower central series of associative algebras}, International Mathematics Research Notices \textbf{2015} (2015) No. 15, 6330--6352. arXiv:1302.3992

\bibitem{Kerchev13}
George Kerchev, \textit{On the filtration of a free algebra by its associative lower central series}, Journal of Algebra \textbf{375} (2013), 322--327. arXiv:1101.5741

\bibitem{Kras97}
A.N. Krasil'nikov, \textit{On the semigroup nilpotency and the Lie nilpotency of associative algebras}, Mathematical Notes \textbf{62} (1997), 426--433.

\bibitem{Kras13}
Alexei Krasilnikov, \textit{The additive group of a Lie nilpotent associative ring}, Journal of Algebra \textbf{392} (2013), 10--22. arXiv:1204.2674.

\bibitem{Latyshev63}
V.N. Latysev, \textit{On the choise of basis in a $T$-ideal}, Sibirskii Matematicheskii Zhurnal (Siberian Mathematical Journal) \textbf{4} (1963) 1122--1127. (in Russian)

\bibitem{Latyshev65}
V.N. Latysev, \textit{On finite generation of a T-ideal with the element $[x_1, x_2, x_3, x_4]$}, Sibirskii Matematicheskii Zhurnal (Siberian Mathematical Journal) \textbf{6} (1965), 1432--1434. (in Russian)

\bibitem{RT99}
D.M. Riley, V. Tasi\'c, \textit{Mal'cev nilpotent algebras}, Archiv der Mathematik (Basel) \textbf{72} (1999), 22--27.

\bibitem{Volichenko78}
I.B. Volichenko, \textit{The $T$-ideal generated by the element $[x_1, x_2, x_3, x_4]$}, Preprint 22, Institute of Mathematics of the Academy os Sciences of the Belorussian SSR, (1978). (in Russian)
\end{thebibliography}
\end{document}